\documentclass[11 pt]{article}  
\usepackage[utf8]{inputenc}
\usepackage{amsmath}
\usepackage{amsfonts}
\usepackage{amssymb}
\usepackage{graphicx}
\usepackage{mathrsfs}
\usepackage{upref,amsthm,amsxtra,exscale}
\usepackage{cite}
\usepackage[colorlinks=true,urlcolor=blue,
citecolor=red,linkcolor=blue,linktocpage,pdfpagelabels,
bookmarksnumbered,bookmarksopen]{hyperref}

\usepackage{fullpage}

\usepackage{subcaption}
\usepackage{caption}
\usepackage{cleveref}

\usepackage{enumitem}

\def\N{\mathbb{N}}
\def\eps{\varepsilon}
\def\R{\mathbb{R}}

\def\cH{\mathcal{H}}

\renewcommand{\O }{\Omega}

\renewcommand{\d }{\delta}
\renewcommand{\l }{\lambda }

\long\def\salta#1{\relax}

\newcommand{\dist}{\mbox{dist}\;}

\newcommand{\ren}{\mathbb{R}^N}

\newcommand{\x}{\times}

\newcommand{\car}{{\raise2pt\hbox{$\chi$}}}

\renewcommand{\d }{\delta }

\renewcommand{\l }{\lambda }

\newcommand{\RR}{\mathbb R}

\renewcommand{\leq}{\leqslant}

\renewcommand{\geq}{\geqslant}

\newtheorem{theorem}{Theorem}[section]
\newtheorem{Corollary}[theorem]{Corollary}

\newtheorem{Lemma}[theorem]{Lemma}

\numberwithin{equation}{section}

\theoremstyle{definition}
\newtheorem{remark}[theorem]{Remark}

\usepackage{color}
\usepackage[dvipsnames]{xcolor}

\title{Uniqueness and nondegeneracy for Dirichlet fractional problems in bounded domains via asymptotic methods}
\author{
Abdelrazek Dieb,
Isabella Ianni\footnote{I. Ianni is  supported by INDAM - GNAMPA,   {\sl Fondi Ateneo - Sapienza},   PRIN  grant $2017$JPCAPN$\_003$, and  VALERE: grant {\sl Vain-Hopes}.},
\& Alberto Saldaña \footnote{A. Saldaña is supported  by 
the 2021 Visiting Professor Programme of La Sapienza University (Italy), by 
CONACYT grant A1-S-10457 (Mexico), and by UNAM-DGAPA-PAPIIT grant IA101721 (Mexico).
}}
\date{}

\begin{document}

\maketitle

\abstract{We consider positive solutions of a fractional Lane-Emden type problem in a bounded domain with Dirichlet conditions. We show that uniqueness and nondegeneracy hold  for the asymptotically linear problem in general domains. Furthermore, we also prove that all the known uniqueness and nondegeneracy results in the local case extend to the nonlocal regime when the fractional parameter $s$ is sufficiently close to 1.}

\section {Introduction  and main results}
We study the uniqueness and the nondegeneracy  of  positive solutions of the Dirichlet Lane-Emden-type fractional problem
\begin{align}\label{main}
    (-\Delta)^{s} u+\l u= u^{p} \quad{\text{ in }} \O,\qquad 
u = 0 \quad {\text{ in }}\RR^N\backslash \O,\qquad 
u> 0 \quad  \text{ in } \O,
\end{align}
where $\Omega$ is a bounded domain of class $C^2$ in $\ren$ ($N\geq 2$), $p>1$, $\lambda\in\mathbb R$, $s\in (0,1)$, and
 $(-\Delta)^s$ denotes the fractional Laplacian given by
\begin{align}\label{operator}
(-\Delta)^{s}u(x):=c_{N,s}\, 
p.v.\int_{\mathbb{R}^{N}}{\frac{u(x)-u(y)}{|x-y|^{N+2s}}\, dy},\qquad c_{N,s}
:=4^{s}\pi^{-\frac{N}{2}}\frac{\Gamma(\frac{N}{2}+s)}{\Gamma(2-s)}s(1-s).
\end{align}
 
A solution $u$ of \eqref{main} is called \emph{nondegenerate} if the linearized problem
 \begin{align}\label{main2}
    (-\Delta)^s v +\lambda v&= pu^{p-1}v\quad \text{ in }\Omega,\qquad v=0\quad \text{ on }\R^N\backslash \Omega,
\end{align}
admits only the trivial solution. 
 
It is immediate to see (by testing the equation with the first Dirichlet eigenfunction of $(-\Delta)^s$ in $\Omega$) that in order to have solutions of \eqref{main} the parameter $\l$ must satisfy $\lambda> -\lambda_1^s(\Omega)$ (here $\lambda_1^s(\Omega)$ is the first Dirichlet eigenvalue of  $(-\Delta)^s$ in $\Omega$). 
Actually, similarly to the local case $s=1$, standard variational methods can be used to yield the existence of solutions of \eqref{main} in any bounded smooth domain $\Omega$ whenever $\lambda> -\lambda_1^s(\Omega)$  and $1<p<\frac{N+2s}{N-2s}$, see, for example, \cite{ServadeiValdinociDCDS2013}.  
For critical ($p=\frac{N+2s}{N-2s}$) and supercritical ($p>\frac{N+2s}{N-2s}$) regimes, existence is more involved.
Indeed, the fractional Pohozaev identity implies the nonexistence of solutions of \eqref{main} if $p\geq \frac{N+2s}{N-2s}$, $\lambda=0$, and if $\O$ is \emph{starshaped}, see \cite{FallWethNonexistence, RosOtonSerraARMA2014}; on the other hand, there are many existence results available in the literature when either $\lambda\neq0$ or the domain is not starshaped, see, for instance, \cite{ServadeiValdinociTAMS2015,ServadeiValdinociCPAA2013,GuoLiPistoiaYanJDE2021, DLS17,HS21, LongYanYangJDE2019}. 
 
 Concerning uniqueness of solutions, this is a difficult problem\textemdash even in the local case \textemdash and it depends in a subtle way on the geometry and topology of the domain $\Omega$, on the value of the parameter $\lambda$, and on the exponent $p$. As we discuss below, this analysis is even harder for nonlocal problems and it has been addressed so far only for ground state solutions of  fractional Schrödinger equations in $\R^N$ in \cite{FS,FV14,FrankLenzmann2013,ChenLiOu2006,LiJEMS2004}.  As far as we know, in the nonlocal case $s\in(0,1)$, the results that we present in this paper are the first to consider uniqueness of solutions in bounded domains.
 
 We remark that the uniqueness of solutions to \eqref{main} does not hold in general, for instance, multiplicity results can be obtained in suitable domains using a Lyapunov-Schmidt reduction argument, see \cite[Theorem 1.2]{DLS17} (see also Remark \ref{rmk:mult} below) for results when $p$ is close to the critical exponent $\frac{N+2s}{N-2s}$ (from below and from above) and $s\in(0,1)$,  see also \cite{LongYanYangJDE2019} for multiplicity results in the critical case (with $\lambda=0$) in domains with shrinking holes.  

Our first theorem shows that all the uniqueness results known for the Laplacian can be extended to the fractional case for $s$ sufficiently close to 1. Let $\lambda_1^1(\Omega)>0$ denote the first Dirichlet eigenvalue for $-\Delta$ in $\Omega$, and let $2^*=\frac{2N}{N-2}$ if $N>2$ and $2^*=\infty$ if $N=2.$ 

\begin{theorem}\label{thm:scto1}
Let $N\geq 2$, $p\in(1,2^*-1)$, $\lambda>-\lambda_1^1(\Omega)$, and $\Omega$ be such that the problem 
\begin{align}\label{c:intro}
    -\Delta u +\lambda u&= u^{p}\quad \text{ in }\Omega,\qquad u=0\quad \text{ on }\partial \Omega,
\end{align}
has a unique positive solution $u$ which is nondegenerate. Then, there is $\sigma=\sigma(\Omega,\lambda,p)\in(0,1)$ such that, for $s\in(\sigma, 1]$, the problem \eqref{main} has a unique solution and it is nondegenerate.
\end{theorem}
 To clarify the reach of Theorem~\ref{thm:scto1}, we review briefly the literature on known results for uniqueness and nonuniqueness of positive solutions to \eqref{c:intro}. 

It is known that uniqueness does not hold in general, indeed  both the geometry and the topology of the domain play a role to obtain multiplicity results.
The typical case when there is more than one positive solution for problem \eqref{c:intro} is when $\O$ is the annulus or, more in  general, a suitable annular shaped domain (see, for example, \cite{CatrinaWangJDE1999, GladialiGrossiPacellaShrikanthCalcVar2011, LiJDE1990, BartschClappGrossiPacellaMA2012}, see also  \cite{EspositoMussoPistoiaJDE2006} for not simply connected planar domains), but uniqueness may fail also in some contractible domains (see   \cite{EspositoMussoPistoiaJDE2006, DancerJDE1988, DancerJDE1990} for dumb-bell domains, even starshaped).

Conversely, if $\Omega$ is a ball, then the solutions of \eqref{c:intro} are radially symmetric as a consequence of the classical  symmetry result by Gidas, Ni, and Nirenberg \cite{GNN1979}; hence, one can get uniqueness  in the local case using ODEs techniques.  In particular, in the case $\lambda =0$ one immediately obtains uniqueness  by scaling arguments which  exploit the homogeneity of the power nonlinearity (see \cite{GNN1979}), while, for $\lambda\neq 0$, the proof turns out to be less direct and the complete result, for the full range of values of  $p$ and $\lambda$ for which existence holds, is contained in several papers 
(see \cite{NiNussbaumCPAM1985, KwongLiTAMS1992, ZhangCPDE1992, SrikanthDiffIntEq1993}), where uniqueness is obtained  by first showing that radial solutions are non-degenerate, thus highlighting a strong relationship between uniqueness and nondegeneracy.   We also mention \cite{AdimurthiYadavaARMA1994} and \cite{AftalionPacellaJDE2003},  where the  uniqueness result in the ball for  the local problem \eqref{c:intro} is reobtained via the Pohozaev identity  and a purely PDE approach, respectively.

In view of these results for the local problem \eqref{c:intro}, we deduce the following corollary of Theorem \ref{thm:scto1}.
\begin{Corollary}\label{corollary:ball}
Let $\Omega$ be a ball of $\mathbb R^N$, $N\geq 2$, $p\in(1,2^*-1)$, $\lambda>-\lambda_1^1(\Omega)$.
Then, there is $\sigma=\sigma(N,\lambda,p)\in(0,1)$ such that, for $s\in(\sigma, 1]$,
the problem \eqref{main} has a unique solution and it is nondegenerate.
\end{Corollary}

Although a general characterization of the domains $\O$ for which uniqueness   and nondegeneracy for problem \eqref{c:intro} hold  is still missing, results are know  also for some  domains different from the ball, from which other applications of Theorem \ref{thm:scto1} can be deduced. 
In \cite{DancerJDE1988, KawohlLectureNotes1985} it is conjectured that the convexity of the domain $\O$ is a sufficient condition to guarantee uniqueness for problem \eqref{c:intro}. This conjecture has been partially proved in dimension $N=2$ and for $\lambda=0$, again showing first that the solutions are nondegenerate. We refer to \cite{LinMM1994}, where it has been proved for \emph{least energy solutions}, and to   \cite{demak},  where recently the proof has been extended to any positive solution for  $p$ sufficiently large, via a delicate computation of  the Morse index which exploits the asymptotic behavior of the solutions.  As a consequence of the result in \cite{LinMM1994}, we state the following asymptotic result for least energy solutions.
\begin{Corollary}\label{c:le}
Let $\Omega\subset \R^2$ be a smooth convex planar domain, $p>1$, $\lambda=0$.
Then, there is $\sigma=\sigma(\Omega,p)\in(0,1)$ such that, for $s\in(\sigma, 1]$,
the problem \eqref{main} has a unique least energy solution and it is nondegenerate.
\end{Corollary}

For general positive solutions, from the results in  \cite{demak}, we derive the following.
\begin{Corollary}
Let $\Omega\subset \R^2$ be a smooth convex planar domain, $\lambda=0$.
Then, there exists $p_0=p_0(\Omega)>1$ such that for any $p>p_0$ there is   $\sigma=\sigma(\Omega,p)\in(0,1)$ such that, for $s\in(\sigma, 1]$,
the problem \eqref{main} has a unique solution and it is nondegenerate.
\end{Corollary}

Let us stress that convexity of the domain is not necessary in order to have uniqueness. Indeed, for problem \eqref{c:intro} some uniqueness and nondegeneracy  results are available in the case $\lambda=0$ also in non-convex domains. We refer to
\cite{DancerMa2003, DamascelliGrossiPacellaAIHP1999,GrossiADE2000} where   domains which are symmetric and convex with respect to $N$ orthogonal directions are considered. In particular, a pure PDE approach based on the maximum principle in dimension $N=2$ is introduced in \cite{DamascelliGrossiPacellaAIHP1999}, while in \cite{GrossiADE2000} the same result is shown to hold in dimension $N\geq 3$ and for almost critical powers, using an asymptotic analysis of the blow-up behavior of the solutions.  From this discussion we can deduce the following consequence of Theorem \ref{thm:scto1}.
\begin{Corollary}
Let $\lambda=0$ and let $\Omega\subset\mathbb R^N$ be a smooth bounded domain symmetric with respect to the hyperplanes $\{x_i=0\}$ and convex in the direction $x_i$ for all $1\leq i\leq N$. 
If either 
\begin{align*}
    \text{$N=2$ and $p>1$}\qquad \text{ or }\qquad 
    \text{$N\geq 3$ and $p+1\in \left(2^*-\varepsilon,2^*\right)$}
\end{align*}
for $\varepsilon>0$ small enough; then, there is   $\sigma=\sigma(\Omega,p)\in(0,1)$ such that, for $s\in(\sigma, 1]$,
the problem \eqref{main} has a unique solution and it is nondegenerate.
\end{Corollary}
Uniqueness for \eqref{c:intro} is also known for domains which are a suitable (even non-convex) perturbation of a ball if $\lambda=0$ \cite{ZouPisa1994}, and for suitable large symmetric domains when $\lambda=1$  \cite{DancerRockyMountain1995}.  Nevertheless, to derive the uniqueness for the nonlocal problem via Theorem \ref{thm:scto1}, the nondegeneracy property for the solutions of \eqref{c:intro} is essential, and this is not shown in \cite{ZouPisa1994,DancerRockyMountain1995}. 

If $\O$ is the unit planar square and $\lambda\neq 0$, then uniqueness results for \eqref{c:intro} are known for specific values of $p$.  The case $p=2$ can be found in \cite{McKennaPacellaPlumRothJDE2009} and $p=3$ in \cite{McKennaPacellaPlumRoth2012}. In both papers, the proof relies on a computer-assisted argument. These results cannot be extended directly via Theorem \ref{thm:scto1} to nonlocal problems, since $\Omega$ has some smoothness restrictions, which are used to guarantee boundary regularity of solutions (see Lemmas \ref{threg} and \ref{lem:unif}). We remark that the proofs in \cite{McKennaPacellaPlumRothJDE2009,McKennaPacellaPlumRoth2012} can possibly be adapted to a smooth perturbation of the square. In any case, we believe that the regularity assumptions in $\Omega$ can be weakened, but in order to make the ideas in our arguments more transparent, we do not pursue this here.  

A particularly interesting case regarding uniqueness of solutions is the asymptotically linear problem, \emph{i.e.} when $p$ is close to $1$, since the domain does not play any role. As a matter of fact, in the local case $s=1$, uniqueness and nondegeneracy hold in any bounded domain, see \cite{DamascelliGrossiPacellaAIHP1999, DancerMa2003, LinMM1994}, 
where the proofs exploit the asymptotics of the positive solutions as $p\rightarrow 1$.  Our next result shows that this is also the case in the nonlocal setting, for \emph{any} $s\in(0,1)$. 

\begin{theorem}\label{exit p small} Let $s\in(0,1)$, $N\geq 2$, $\O\subset \R^N$ be a bounded domain of class $C^2$, and $\l>-\l_1^s(\O)$. There is $p_0=p_0(\O, s,\lambda)>1$ such that problem \eqref{main} has a unique solution for every $p\in(1,p_0)$. Moreover, the solution is non-degenerate.
\end{theorem}

As we mentioned before, showing uniqueness results in the nonlocal case is particularly difficult and challenging, even in seemingly simple cases.  For instance, let $\lambda=0$ and let $\Omega$ be a ball. Then, a moving plane argument shows that all the solutions of \eqref{main} are radially symmetric (see, \emph{e.g.}, \cite{JW14}); however, a general uniqueness result for this problem with $s\in(0,1)$ is still open and Corollary \ref{corollary:ball} and Theorem \ref{exit p small} are the first  results in this direction.  The main obstacle when trying to extend the methods used  in the local case to the fractional regime is the lack of essential tools such as shooting methods and other ODE techniques.  Moreover, Courant's nodal theorems, Hopf Lemmas for sign-changing solutions, and monotonicity formulas in bounded domains are also not available, which complicates the analysis of the linearized associated problem to understand the nondegeneracy properties of solutions.  Our results provide a first answer to these questions, showing the uniqueness and nondegeneracy of solutions for either $s$ close to 1 or $p$ close to 1. 

The proofs of Theorems \ref{thm:scto1} and \ref{exit p small} rely on a delicate asymptotic analysis that involves several elements such as interior and boundary regularity  estimates, precise a priori bounds which are uniform in $p$ and $s$ (see Remark \ref{rmk:c}), Hopf Lemmas for positive solutions, Liouville type theorems for linear and nonlinear nonlocal equations in unbounded domains, and the known results for the local problem ($s=1$) and for the linear equation ($p=1$). A similar approach is used in \cite{FV14}, where an asymptotic analysis as $s\rightarrow 1$ together with the Caffarelli-Silvestre extension is used to characterize the uniqueness and nondegeneracy of ground state solutions of a fractional Schrödinger equation in $\R^N$. We emphasize that our proofs do not require extension operators. In particular, we believe that our proofs can be adapted to consider more general nonlocal operators as well.  Furthermore, we think that Theorems \ref{thm:scto1} and \ref{exit p small} could be a first step towards a more complete answer on the uniqueness properties of solutions for all $s\in(0,1)$ and general values of $p$, by using the implicit function theorem (as in \cite[Theorem 4.4]{DamascelliGrossiPacellaAIHP1999}, since the $p$ dependence on $\sigma$ in Theorem~\ref{thm:scto1} can be relaxed, see Remark \ref{rmk:p:dep}) and by showing the nondegeneracy of solutions in this range.

Finally, a natural question is if a similar uniqueness result holds for more general nonlinearities. We emphasize that our proofs use different blow up arguments that take advantage on the particular homogeneity of the power nonlinearity. However, it should be noted that uniqueness may not hold for general superlinear nonlinearities; for instance, multiplicity of positive solutions is known in the case of concave-convex nonlinearities (in any bounded domain), see, for instance, \cite{BarriosColoradoServadeiSoriaHIHP2015}. On the other hand, general uniqueness results for fractional sublinear type problems are available, see \emph{e.g.} \cite[Theorem 6.1 \& Corollary 6.3]{BFSST18}, and for large solutions (blowing-up at the boundary), see for example \cite{abatangelo2015large,MR4062980}.

The organization of the paper is as follows.  In Section \ref{sec:2} we detail our functional setting, include some interior and boundary regularity results, and show some Liouville-type results. In Section \ref{sec:pcto1} we present the proof of Theorem~\ref{exit p small} and the proof of Theorem~\ref{thm:scto1} can be found in Section \ref{sec:l}.

\section{Preliminaries}\label{sec:2}

\subsection{Functional setting and notation}

In this section we introduce some notations and preliminary results needed throughout this paper.

Let $N\geq 2$ and let $\Omega\subset \R^N$ be a bounded domain of class $C^2$. For $p\in[1,\infty]$, we use $\|\cdot\|_{L^p(\Omega)}$ to denote the norm in $L^p(\Omega)$, namely,
\begin{align*}
    \|u\|_{L^p(\Omega)}:=\left(\int_{\Omega}|u|^p\, dx\right)^\frac{1}{p}\quad \text{ for }p\in[1,\infty),\qquad 
    \|u\|_{L^\infty(\Omega)}:=\sup_{\Omega}|u|.
\end{align*}
We sometimes omit $\Omega$ if $\Omega=\R^N$, namely, $\|u\|_{L^p}:=\|u\|_{L^p(\R^N)}$.

Let $s\in (0,1)$, we define $H^{s}(\ren)$ the classical fractional Sobolev space,
\begin{align}\label{}
H^s(\ren)=\left\lbrace u \in
L^2(\ren):\frac{|u(x)-u(y)|}{|x-y|^{\frac{d}{2}+s}} \in
L^2(\ren\times\ren )\right\rbrace ,
\end{align}
endowed with the norm
\begin{align*}
\|u\|^2_{s}:=
\int_{\R^N}|u|^2\, dx + \dfrac{c_{N,s}}{2}\iint_{\ren\times \ren}\dfrac{|u(x)-u(y)|^{2}}{|x-y|^{N+2s}}dx\,dy.
\end{align*}
For a given bounded domain $\Omega\subset \R^N$, let
$$
\cH^s_0(\Omega)=\left\{ u \in H^{s}(\ren):\, \, u=0 \, \text{ in } \ren\backslash \Omega \right\},
$$
Notice that $\cH^s_0(\Omega)$ is a Hilbert space with the norm $\|\,\cdot\,\|_s$.  Moreover, for all $u,\,v\in\cH^s_0(\Omega)$,
\begin{align}\label{bypart}
\int_\Omega(-\Delta)^s uv\, dx =\frac{c_{N,s}}{2} \iint_{\ren\times\ren}{\frac{(u(x)-u(y))(v(x)-v(y))}{|x-y|^{N+2s}}\,dx\,dy}.
\end{align}

We also use $H^1_0(\Omega)$ to denote the usual Sobolev space of weakly differentiable functions with zero trace. 

For $m\in \N_0$, $\sigma\in{(0,1]}$, $s=m+\sigma$, we write $C^s(\Omega)$ (resp. $C^s(\overline{\Omega})$) to denote the space of $m$-times continuously differentiable functions in $\Omega$ (resp. $\overline{\Omega}$) whose derivatives of order $m$ are locally $\sigma$-H\"older continuous in $\Omega$ {(or Lipschitz continuous if $\sigma=1$)}.  We use $[\,\cdot\,]_{C^\sigma(\Omega)}$ for the Hölder seminorm in a domain $\Omega$, namely, 
\begin{align*}
        [u]_{C^\sigma(\Omega)}:=\sup_{\substack{x,y\in \Omega\\x\neq y}}\frac{|u(x)-u(y)|}{|x-y|^\sigma}
\end{align*}
and $\|u\|_{C^s(\Omega)}:= \sum_{|\alpha|\leq m}\|\partial^\alpha u\|_{L^\infty(\Omega)}+\sup_{|\alpha|=m}[\partial^\alpha u]_{C^\sigma(\Omega)}$ is the usual Hölder norm in $C^s(\Omega)$.

\subsection{Regularity results}
We say that $u \in \cH^s_0(\O)$ is a weak solution of $(-\Delta)^s u=f$ in $\Omega$ and $u=0$ in $\R^N\backslash\Omega$ if
\begin{align}\label{eng sol}
\frac{c_{N,s}}{2}\iint_{\ren\times\ren}\frac{ \big(u(x)-u(y)\big)		\big(\varphi(x)-\varphi(y)\big)}{|x-y|^{N+2s}}\,dx\,dy=\int_\O f\varphi
\quad \text{ for all } \varphi \in \cH^s_0(\Omega).
\end{align}

\begin{remark}
A standard Moser iteration and bootstrap argument (see, for instance, \cite[Proposition 8.1]{ROSV17}) readily implies that a weak solution $u$ of \eqref{main} and a weak solution $v$ of \eqref{main2} are smooth in the sense that $u,v\in C^\infty(\Omega)\cap L^\infty(\Omega)\cap C^s(\R^N)$.
\end{remark}

In our proofs via asymptotic methods as $p$ or as $s$ goes to 1, it is crucial to have uniform estimates (sometimes in $p$ and sometimes in $s$) on solutions.  However, many classical regularity estimates do not emphasize explicitly the dependencies of the constants, which is not always straightforward (see Remark \ref{rmk:c}). In this section, for completeness, we revisit some classical regularity estimates to remark the dependencies of the constants involved in some known inequalities.

\begin{Lemma}\label{threg}
Let $\Omega\subset \R^N$ be a bounded domain of class $C^2$ and let $u$ be a solution of $(-\Delta)^s u =g$ in $\Omega$ with $g\in L^\infty(\Omega)$ and $u=0$ in $\R^N\backslash \Omega$. There is $C=C(\O)>0$ such that
\begin{align}\label{boundUpToBoundary}
\|u\|_{C^s(\overline\O)}\leq C  \|g\|_{L^\infty(\Omega)},
\end{align}
Furthermore, $\frac{u}{\delta^s}\in C^{\alpha}(\bar\Omega)$ for some $\alpha\in (0,1)$ and \begin{align*}
\left\|\frac{u}{\delta^s}\right\|_{C^\alpha(\overline\O)}\leq C_1 \|g\|_{L^\infty(\Omega)},
\end{align*}
for some constant $C_1=C_1(\O,s)>0$.
\end{Lemma}
\begin{proof}
The interior and boundary regularity with follows from \cite[Proposition 1.1,  Theorem 1.2 \& Proposition 1.4]{reguros}.  The fact that the constant $C=C(\Omega)>0$ is independent of $s$, follows from \cite[Lemma 3.4]{JSW20}.
\end{proof}

\begin{Lemma}\label{lemma:interiorRegularity1}
Let $\text{$u\in C^{2s}(B_{3R})$}\cap C(\R^N)\cap L^\infty(\R^N)$ with compact support in $\R^N$, then for any $\alpha\in (0,2s)$ there exists a constant $C:=C(s,d,\alpha, R)>0$ such that
\[\|u\|_{C^\alpha(B_R)}\leq C\left( \|u\|_{L^{\infty}(\R^N)}+ \|(-\Delta)^s u\|_{L^{\infty}(B_{2R})} \right).\]
\end{Lemma}
\begin{proof}Let $\alpha'\in(\alpha, 2s)$, $f=(-\Delta)^s u$ in $B_{2R}$, $\eta_{\epsilon}$ be the standard mollifier. Then  $(-\Delta)^s (u\ast \eta_{\epsilon}) =f\ast \eta_{\epsilon}$ in $B_{2R'}$ for $R'=R+\epsilon/2$ and $\epsilon$ small. Thus, by \cite[Proposition 2.3]{reguros},
\[\|u\ast \eta_{\epsilon}\|_{C^{\alpha'}(B_{R'})}\leq C\left( \|u\ast \eta_{\epsilon}\|_{L^{\infty}(\R^N)}+ \|f\ast \eta_{\epsilon}\|_{L^{\infty}(B_{2R'})} \right).\]
The claim now follows passing to the limit as $\eps\to 0$.
\end{proof}
\begin{Lemma}\label{lemma:interiorRegularity2}
Let $\alpha\in (0,2s)$, $u\in C^{2s+\alpha}(B_{3R})\cap C(\R^N)\cap L^\infty(\R^N)$ with compact support in $\R^N$, then there exists a constant $C:=C(s,d,\alpha, R)>0$ such that

\[
\|u\|_{C^{\alpha+2s}(B_R)}\leq C\left( \|u\|_{L^{\infty}(\R^N)}+ \|(-\Delta)^s u\|_{C^{\alpha}(B_{2R})} \right).
\]
\end{Lemma}
\begin{proof}
Similar to the proof of Lemma~\ref{lemma:interiorRegularity1}, but using \cite[Proposition 2.2]{reguros}.
\end{proof}

\begin{Lemma}\label{lem:int:reg}
Let $r>0$, $s\in(0,1)$, and $u\in L^\infty(\R^N)$. There is $C=C(N)$ such that
\begin{align*}
    [u]_{C^s(B_{r/2})}
    \leq Cr^s\left( \|(-\Delta)^s u\|_{L^\infty(B_r)}+\tau_{N,s}\int_{\R^N\backslash B_r}\frac{|u(z)|}{|z|^N(|z|^2-r^2)^s}\, dz \right),
\end{align*}
where $\tau_{N,s}=\frac{2}{\Gamma(s)\Gamma(1-s)|S^{N-1}|}$.
\end{Lemma}
\begin{proof}
The proof can be found in \cite[Lemma A.1]{JSW20}.
\end{proof}

The following Lemma is an easy consequence of the barrier constructed in \cite[Lemma 2.6]{reguros}.  However, we use this result in a blow up argument as $s\to 1^-$ and as the domain grows; therefore, we require that the constants involved in the construction of the barrier do not degenerate in the limit.  We include a proof for completeness. 

\begin{Lemma}\label{lem:unif}
Let $\Omega\subset \R^N$ be a (possibly unbounded) domain satisfying a uniform sphere condition, namely, there is $r_0>0$ such that, for every $x_0\in \partial \Omega$, there is $y_0\in \R^N\backslash \Omega$ with $\overline{B_{r_0}(y_0)}\cap \overline{\Omega} = \{x_0\}$. Let $\sigma\in(0,1)$, $s\in(\sigma,1)$, $u\in L^\infty(\R^N)$ be a pointwise solution of $(-\Delta)^su=f$ in $\Omega$ with $f\in L^\infty(\Omega)$, and $u=0$ in $\R^N\backslash \Omega$. Let $M>0$ satisfy that $\|u\|_{L^\infty}+\|f\|_{L^\infty(\Omega)}<M$.  Then, there are $C=C(N,M,r_0,\sigma)>0$ and $\delta= \delta(N,M,r_0,\sigma)>0$ such that
\begin{align*}
    |u(x)|<C \operatorname{dist}(x,\partial \Omega)^s\qquad \text{for all $x\in \Omega$ with $\operatorname{dist}(x,\partial \Omega)<\delta$.}
\end{align*}
In particular, $u\in C^s(\overline{\Omega})$, $u=0$ on $\partial \Omega$, and there is $\delta_0=\delta_0(N,M,r_0,\sigma)>0$ such that $|u(x)|<\frac{1}{2}$ if $\operatorname{dist}(x,\partial \Omega)<\delta_0$.
\end{Lemma}
\begin{proof}
By contradiction, assume that there is $(x_n)\subset \Omega$ and $(s_n)\subset (\sigma,1)$ such that
\begin{align}\label{con2}
 \text{$\operatorname{dist}(x_n,\partial \Omega)\to 0$ as $n\to\infty$ and $|u(x_n)|\geq n\operatorname{dist}(x_n,\partial \Omega)^{s_n}$ for all $n\in\N$.}
\end{align} Let $\xi_n\in \partial \Omega$ be such that $\operatorname{dist}(x_n,\partial \Omega)=|x_n-\xi_n|$.  By the uniform sphere condition, 
\begin{align*}
B_{r_0}(y_n)\subset \R^N\backslash \Omega,
\qquad \overline{B_{r_0}(y_n)}\cap \overline{\Omega}=\{\xi_n\},
\qquad y_n=-\frac{x_n-\xi_n}{|x_n-\xi_n|}r_0.
\end{align*}
By scaling the problem, we can assume that $r_0=1.$ Now we use the torsion function to build a barrier, as in \cite[Lemma 2.6]{reguros}. Let
\begin{align*}
    \psi_n(x):=a_{N,{s_n}}(1-|x-y_n|^2)^{s_n}_+,\qquad a_{N,{s_n}}:=\frac{2^{-2{s_n}}\Gamma(\tfrac{N}{2})}{\Gamma(\frac{N+2{s_n}}{2})\Gamma(1+{s_n})}.
\end{align*}
Then $(-\Delta)^{s_n} \psi_n=1$ in $B_1(y_n)$. Let $\zeta$ be a fixed radial smooth function such that $\zeta=M$ in $\R^N\backslash B_3$ and $\zeta=0$ in $B_2$. Let
\begin{align*}
    M_1:=\sup_{s\in(\sigma,1),x\in B_3}|(-\Delta)^s\zeta(x)|,
    \qquad M_0:=3^{2+N}(M+M_1),
\end{align*}
note that $M_1<\infty$, \emph{e.g.}, by \cite[Lemma B.5]{AJS18}.

 Let $D:=B_3(y_n)\cap \Omega$. Using the Kelvin transform \cite[Proposition A.1]{reguros}, we have that
\begin{align*}
\Psi_n(x):=M_0 a_{N,{s_n}}|x-y_n|^{2{s_n}-N}(1-|x-y_n|^{-2})^{s_n}_+ + \zeta(x-y_n),\qquad x\in\R^N,
\end{align*}
satisfies that
\begin{align*}
(-\Delta)^{s_n} \Psi_n(x)
=M_0|x-y_n|^{-2{s_n}-N}+(-\Delta)^{s_n}\zeta(x-y_n)
\geq \frac{M_0}{3^{2+N}}-M_1
= M\geq (-\Delta)^{s_n} u(x)
\end{align*}
in $D$, $\Psi_n=0$ in $B_1(y_n)$, $\Psi_n\geq M\geq u$ in $\R^N\backslash B_3(y_n)$, and $\Psi_n\in H^{s_n}_{loc}(\R^N)$ (see \cite[Lemma 2.6]{reguros}). Then, since $u=0$ in $\R^N\backslash \Omega$, we have that
$\Psi_n-u\geq 0$ in $\R^N\backslash D.$  Therefore, by the weak maximum principle (see for example \cite[Proposition 3.1 \& Remark 3.2]{FJ15}), $u(x)\leq \Psi_n(x)$ in $D$. Arguing similarly with $-u$, we obtain also that $|u(x)|\leq \Psi_n(x)$ in $D$. Since $|x_n-\xi_n|=\dist(x_n,\partial \Omega)\to 0$, we may assume that $x_n\in \Omega \cap B_{2}(y_n)$.  Then, using that $\zeta=0$ in $B_2(y_n)$ and that $|x_n-y_n|=1+|x_n-\xi_n|=1+\operatorname{dist}(x_n,\partial \Omega)>1$, we obtain that
 \begin{align}
     u(x_n)&\leq M_0 a_{N,{s_n}}|x_n-y_n|^{2{s_n}-N}(1-|x_n-y_n|^{-2})^{s_n}
     \leq M_0 a_{N,{s_n}}\left(|x_n-y_n|^2-1\right)^{s_n}\notag\\
     \label{eqGamm}
     &= M_0 a_{N,{s_n}}\left(|x_n-y_n|+1\right)^{s_n}\left(|x_n-y_n|-1\right)^{s_n}
     \leq 3 M_0C_1\dist(x_n,\partial \Omega)^{s_n}
     \end{align}
 for $x\in \Omega \cap B_{2}(y_n)$, where $C_1:=\sup_{t\in(\sigma,1)} a_{N,t}<\infty$, by the properties of the $\Gamma$ function.
Since \eqref{eqGamm} contradicts \eqref{con2}, the claim follows. 
\end{proof}

\begin{remark}\label{rmk:bdry}
 Let $\Omega$ be a domain with $C^2$ boundary and $\sigma\in(0,1)$. Then there is $r_0>0$ such that $\Omega$ satisfies the uniform exterior sphere condition. Let $\delta_0=\delta_0(N,M,r_0,\sigma)>0$ by given by Lemma \ref{lem:unif}. Note that, if  $\mu>1$, then the domain $\mu\Omega$ satisfies in particular the uniform exterior sphere condition \emph{with the same $r_0$}. Therefore, for any $\eta>1$ and $s\in(\sigma,1)$, if $u_{s,\eta}\in L^\infty(\R^N)$ satisfies that $u_{s,\eta}=0$ in $\R^N\backslash (\eta\Omega)$ and 
 $\|u_{s,\eta}\|_{L^\infty}+\|(-\Delta)^su_{s,\eta}\|_{L^\infty(\eta\Omega)}<M$ for some $M$ independent of $\eta$ and $s$, then $|u_{s,\eta}(x)|<\frac{1}{2}$ for all $x\in \eta\Omega$ such that $\operatorname{dist}(x,\partial (\eta\Omega))<\delta_0$.
\end{remark}

\subsection{Liouville theorems}
We show the nonexistence of positive bounded solutions for a linear fractional problem both in the whole space and in the half space. We do not require integrability assumptions on the solutions.

We denote by $\lambda_1^s(B_R)$ and $\varphi_{1,R}^s$ the first eigenvalue and the corresponding positive eigenfunction of $(-\Delta)^s$ in the ball $B_R$ of radius $R>0$ centred at the origin, with zero Dirichlet exterior condition.  Let us recall that  $\lambda_1^s(B_R)\to 0$ as $R\to \infty$ (for instance, by the fractional Faber-Krahn inequality \cite{BrascoLindgrenParini}).

\begin{theorem}\label{Theorem:BigLiouville}
Let $R>0$ be such that $\lambda_1^s(B_R)<1$. If 
 $w\in L^{\infty}(\mathbb R^N)$ satisfies pointwisely that
\begin{align}\label{eq:bigLiuov}
\left\lbrace 
\begin{array}{rcll}
(-\Delta)^s w & = & w & {\text{ in }}B_{R}(x_0),\\
w & \geq & 0 &{\text{ in }} \ren.
\end{array}\right.
\end{align}
for some $x_0$ in $\mathbb R^N$, then $w\equiv 0$ in $\mathbb R^N$.
\end{theorem}

\begin{Corollary}\label{liouv1}
Let $w$ be a bounded classical solution to
\begin{align*}
\left\lbrace 
\begin{array}{rcll}
(-\Delta)^s w & = & w & {\text{ in }} \ren,\\
w & \geq & 0 &{\text{ in }} \ren.
\end{array}\right.
\end{align*}
Then $w=0$ in $\ren$.
\end{Corollary}
\begin{Corollary}\label{liouv2}
Let $w$ be a bounded classical solution to
\begin{align}\label{liouv}
\left\lbrace 
\begin{array}{rcll}
(-\Delta)^s w & = & w & {\text{ in }} \ren_+,\\
w & \geq & 0 &{\text{ in }} \ren.
\end{array}\right.
\end{align}
Then $w=0$ in $\ren_+$.
\end{Corollary}
Corollary \ref{liouv1} is a special case of a Liouville type result proved by Jin, Li and Xiong \cite{JinLiXiong} for  more general, also nonlinear, problems (see also \cite{BrandleColoradoPabloSanchez, Ou, ChenLiOu}). Our proof is very simple but applies only in the linear case.
For the half space, non-existence results for \eqref{liouv} may be found in  \cite{FallWethNonexistence} but under integrability assumptions, see also  \cite{FallWethMonotonicity} for superlinear problems.

\begin{proof}[Proof of Theorem~\ref{Theorem:BigLiouville}]
By contradiction, assume there is a bounded nontrivial (pointwise) solution $w$ of \eqref{eq:bigLiuov}. Let $\phi(x):=\varphi_{1,R}^s(x-x_0)$.  Note that
\begin{align*}
    \frac{c_{N,s}}{2}\int_{\R^N}
\int_{\mathbb{R}^{N}}\frac{(\phi(x)-\phi(y))(w(x)-w(y))}{|x-y|^{N+2s}}\, dy
\,dx<\infty,
\end{align*}
by \cite[Lemma 2.2]{fallmorse}. Moreover, using that $\phi\leq C(1-|x-x_0|^2)^s$ and \cite[Lemma 2.3]{JSW20} (with $N\geq 2$, $a=-s$, $\lambda=-2s$), there is a constant $C_1=C_1(N,R,s)>0$ such that
\begin{align*}
0&\leq \int_{\R^N\backslash B_R(x_0)}w(x) \int_{B_R(x_0)}\frac{\phi(y)}{|x-y|^{N+2s}}\, dy\,dx
\leq C\|w\|_{L^\infty}\int_{\R^N\backslash B_R(x_0)} \int_{B_R(0)}\frac{(1-|y|^2)^s}{|x-y|^{N+2s}}\, dy\,dx\\
&\leq CC_1|w|_{L^\infty}\int_{\R^N\backslash B_R(x_0)}\frac{1+(1-|x|^2)^{-s}}{1+(1-|x|^2)^{N+2s}} \,dx<\infty.
\end{align*}
Therefore, using that $(-\Delta)^s\phi = \lambda_1^s(B_R)\phi$ in $B_R(x_0)$, $w\geq 0$, and $\phi\geq 0$ in $\R^N$,
\begin{align*}
\int_{B_R(x_0)}\phi\, w\,dx
&=\int_{\R^N}\phi(-\Delta)^s w\,dx 
=c_{N,s}\int_{\R^N}
p.v.\int_{\mathbb{R}^{N}}\phi(x)\frac{w(x)-w(y)}{|x-y|^{N+2s}}\, dy
\,dx\\
&=\frac{c_{N,s}}{2}\int_{\R^N}
\int_{\mathbb{R}^{N}}\frac{(\phi(x)-\phi(y))(w(x)-w(y))}{|x-y|^{N+2s}}\, dy
\,dx\\
&=c_{N,s}\int_{\R^N}
p.v.\int_{\mathbb{R}^{N}}w(x)\frac{\phi(x)-\phi(y)}{|x-y|^{N+2s}}\, dy
\,dx \\
&=\lambda_1^s(B_R)\int_{B_R(x_0)}w \phi\,dx 
-\int_{\R^N\backslash B_R(x_0)}w(x) \int_{B_R(x_0)}\frac{\phi(y)}{|x-y|^{N+2s}}\, dy\,dx\\
&<\int_{B_R(x_0)}\phi\, w\,dx,
\end{align*}
which yields a contradiction.  
\end{proof}

\begin{remark}[Multiplicity of solutions]\label{rmk:mult}
A concrete multiplicity result for \eqref{main} can be deduced from \cite[Theorem 1.2]{DLS17} in the following way.  Let $s\in(0,1)$, $\Omega:=U\cup V\subset \R^N$, $N\geq 2$, where $U$ and $V$ are two disjoint balls of different sizes. Let $G$ denote the Green's function for $(-\Delta)^s$ in $\Omega$ and, for $\xi\in \Omega$, let $H$ be the solution of 
\begin{align*}
(-\Delta)^s H(x,\xi)=0\quad \text{ for }x\in\Omega,\qquad H(x,\xi)=\Gamma(x-\xi)\quad \text{ for }x\in \R^N\backslash \Omega,
\end{align*}
 where $\Gamma$ is the fundamental solution of $(-\Delta)^s$ in $\R^N$. The function $H$ is the regular part of the Green's function. Let
\begin{align*}
\Psi(\xi,\Lambda):=\frac{H(\xi,\xi)}{2}\Lambda^2-\log(\Lambda),\qquad  \xi\in \Omega,\ \Lambda\in (0,\infty).
\end{align*}
In particular, since $H(\xi,\xi)$ is positive for $\xi\in\Omega$ and $H(\xi,\xi)\to \infty$ as $\xi\to \partial\Omega$ (because $\Gamma(x)\to \infty$ as $|x|\to 0$), it is easy to see that $\Psi$ has two local minima at $(\xi_1,\Lambda_1)$ and $(\xi_1,\Lambda_2)$ for some $\xi_1\in U$ and $\xi_2\in V$.  These points are stable critical points of $\Psi$.  Therefore, by \cite[Theorem 1.2]{DLS17}, there are two different positive solutions of \eqref{main} with $p=\frac{N+2s}{N-2s}-\eps$ with $\eps>0$ small enough. 
\end{remark}

\section{The asymptotically linear case}\label{sec:pcto1}
In this section we prove Theorem~\ref{exit p small}. Consider the linearized problem at a solution $u$ of \eqref{main} given by
\begin{equation}
\label{linearized}
\left\lbrace 
\begin{array}{rcll}
(-\Delta)^{s} h+\l h &= & pu^{p-1}h & {\text{ in }} \O,\\
h &= & 0 & {\text{ in }}\RR^N\backslash \O, \\
\end{array}\right.
\end{equation}
and recall that a solution $u$ of \eqref{main} is said to be \emph{nondegenerate} if \eqref{linearized} has only the trivial solution $h=0$, i.e. if $\mu=0$ is not an eigenvalue of the linearized operator $L_u:=(-\Delta)^{s} +\l - pu^{p-1} $.

We prove first the following key lemma.
\begin{Lemma}\label{pclose1}
Let $s\in(0,1)$, $(p_n)_n\subset(1,\frac{N+2s}{N-2s})$ be a sequence such that $p_n\to 1$, let $u_n$ be a solution to \eqref{main} with $p=p_n$, and let $M_n:=\|u_n\|_{L^\infty}$. Then,
\begin{align*}
M_n^{p_n-1}\to\l_1^s(\O)+\l \ \text{  and  }\ \frac{u_n}{M_n}\to \varphi_1^s \text{ uniformly in $\overline\O $  and in $C^{2s+\alpha}_{loc}(\O)$ as $n\to\infty$}  
\end{align*}
for some $\alpha\in(0,1)$, where $\lambda_1^s(\O)$ and  $\varphi_1^s$ denote respectively the first eigenvalue and eigenfunction for the fractional Laplacian in $\Omega$ with exterior Dirichlet condition.
\end{Lemma}

\begin{proof}
{\sl Step 1. We show that $M_n^{p_n-1}$ is bounded.}\\
By contradiction, assume that $M_n^{p_n-1}\to \infty$ and let us 
perform a blow-up argument.
Let $(x_n)_{n\geq 1}$ be a sequence in $\O$ such that $M_n=u_n(x_n)$. Define
\begin{align*}
v_n(y):=\frac{1}{M_n}u_n\left( \mu_ny+x_n\right) 
\end{align*}
where $\mu_n:=M_n^{\frac{1-p_n}{2s}}\to 0$, then $v_n$ is a function satisfying $0 \leq v_n\leq 1$, $v_n(0) = 1$ and
\begin{align}\label{v_n eq}
\left\{\begin{array}{lr}
(-\Delta)^s v_n=v_n^{p_n}-\frac{\l}{M_n^{p_n-1}}v_n=:f_n\quad &\text{ in } \O_n,\\
 v_n=0 & \text{ in }\ren\backslash \O_n,
 \end{array}\right.
\end{align}
where $\O_n=\left\{y\in\ren:\, \mu_ny+x_n\in \O \right\}$. Up to subsequences, two situations may occur: 
\begin{align}\label{2cases}
\text{  either }\quad  \operatorname{dist}\left(x_{n},\,\partial\O\right) \mu_{n}^{-1} \rightarrow+\infty \quad \text{ or } \quad \operatorname{dist}\left(x_{n},\,\partial\O\right) \mu_{n}^{-1} \rightarrow \rho \geq 0.
\end{align}
Assume the first case holds, so that $\Omega_n \rightarrow \ren$ as $n \rightarrow+\infty$. 

We claim that, for any $R>0$ and $\alpha\in(0,\min\{2s,1\})$, there exists $n_R\in\mathbb N$ and $C=C(s,N,\alpha,R)>0$
such that
\begin{align}
\label{vnlocReg}
v_n\in C^{2s+\alpha}(\overline B_R) \ \mbox{ and }\ \|v_n\|_{C^{2s+\alpha}(\overline B_R)}\leq C, \quad \forall n\geq n_R.
\end{align}
In order to prove \eqref{vnlocReg} let us fix $R>0$. Then, since $\Omega_n \rightarrow \ren$ as $n \rightarrow+\infty$, there exists $n_R\in\mathbb N$ such that $\overline B_{4R}\subset\O_n$ for any $n\geq n_R$.
Since $v_n\in C^s(\ren)$ satisfies the uniform bound $0\leq v_{n} \leq 1$, then $f_n\in L^\infty(\overline B_{4R})$
with uniform bounds so, as a consequence of Lemma~\ref{lemma:interiorRegularity1}, $v_n\in C^{\alpha}(\overline B_{2R})$, for any $\alpha \in (0,2s)$, and \\
\[\|v_n\|_{C^\alpha(B_{2R})}\leq C,\]
where the constant $C=C(s,N,\alpha,R)$.
Then $f_n\in C^{\alpha}(B_{2R})$ with uniform $C^{\alpha}$-bound and so \eqref{vnlocReg} follows by Lemma~\ref{lemma:interiorRegularity2}.

By \eqref{vnlocReg}, using Arzel\`{a}-Ascoli's theorem and a diagonal argument we obtain that there exists a function $v\in C^{2s+\frac{\alpha}{2}}(\ren)$ such that,  passing to a subsequence,  $v_{n} \rightarrow v$ in $C^{2s+\frac{\alpha}{2}}_{loc}(\ren)$. Passing to the limit in 
\eqref{v_n eq} we see that $v$ solves $(-\Delta)^{s} v=v$ in $\ren$ pointwisely.
Furthermore, $0\leq v\leq 1$, so $v$ is bounded in $\ren$ and it follows that $v>0$ in $\R^N$; indeed, if $v(x_0)=0$ for some $x_0\in\R^N$, then
\begin{align*}
0=v(x_0)=c_{N,s}p.v.\int_{\R^N}\frac{-u(y)}{|x_0-y|^{N+2s}}\, dy<0.
\end{align*}
However, by Corollary \ref{liouv1}, no such $w$ can exist and we have reached a contradiction. 

If the second case in \eqref{2cases} holds then we may assume $x_{n} \rightarrow x_{0} \in \partial \Omega .$ With no loss of generality assume also $\nu(x_{0})=-e_{N}$. In this case it will be more convenient to work with
$$
w_{n}(y):=\frac{u_{n}\left(\mu_{n} y+\xi_{n}\right)}{M_n},
$$
where $\xi_{n} \in \partial \Omega$ is the projection of $x_{n}$ on $\partial \Omega$ and let $D_{n}:=\left\{y \in \mathbb{R}^{N}: \mu_{n} y+\xi_{n} \in \Omega\right\}.$  Observe that
\begin{align}
\label{0inD}
0 \in \partial D_{n}
\end{align}
and $D_{n} \rightarrow \R^N_+:=\left\{y \in \mathbb{R}^{N}: y_{N}>0\right\}$ as $n \rightarrow+\infty$.  It also follows that $w_{n}$ satisfies that
\begin{align}\label{v_n eq2}
(-\Delta)^s w_n=w_n^{p_n}-\frac{\l}{M_n^{p_n-1}}w_n \quad \text{ in } D_n,\qquad w_n=0\quad  \text{ in }\ren\backslash D_n.
\end{align}
Moreover, setting
$$
y_{n}:=\frac{x_{n}-\xi_{n}}{\mu_{n}}\ (\in D_n)
$$
so that $\left|y_{n}\right|=\operatorname{dist}\left(x_{n},\,\partial\O\right) \mu_{n}^{-1},$ and $w_{n}(y_{n})=1$. We claim  that
\begin{align}\label{claim:rhopositive}
\rho=\lim _{n \rightarrow+\infty} \operatorname{dist}\left(x_{n},\,\partial\O\right) \mu_{n}^{-1}>0.
\end{align}
 Observe that this will allow us to conclude that, up to passing to a further subsequence, $y_{n} \rightarrow y_{0}$ where $\left|y_{0}\right|=\rho>0,$ thus $y_{0}$ is an interior point of the half-space $\R^N_+$. Let us now show the claim. Observe that
  $\|w_n\|_{\infty}+\|(-\Delta)^sw_n\|_{L^{\infty}(D_n)}\leq 3=:M$, for  $n$ sufficiently large, hence by Lemma \ref{lem:unif} and Remark \ref{rmk:bdry}, there exists $\delta_0>0$ such that
 \begin{align}
\label{ah}w_n(y)<\frac{1}{2},\quad\mbox{ for all }y\in D_n \mbox{ such that }\operatorname{dist}\left(y,\partial D_n\right)<\delta_0,\end{align}
 for $n$ sufficiently large.  Now, by contradiction, assume that 
 \begin{align}
 \label{B_up_cont}
 \lim _{n \rightarrow+\infty} \operatorname{dist}\left(x_{n},\,\partial\O\right) \mu_{n}^{-1}=0.
 \end{align}
Then, since  by \eqref{0inD} we have $\operatorname{dist}\left(y_{n},\partial D_n\right)\leq\left|y_{n}\right|$, it follows that $\operatorname{dist}\left(y_{n},\partial D_n\right)<\delta_0$ for $n$ large. As a consequence, \eqref{ah} implies that $1 =w_n(y_n)<\frac{1}{2}$, which gives a contradiction  and  proves the claim in \eqref{claim:rhopositive}. \\
Now, arguing similarly as in the first case, we obtain that $w_{n} \rightarrow w$ in $C^{2s+\frac{\alpha}{2}}_{loc}(\ren_{+})$, where $w$ satisfies that $0 \leq w \leq 1$ in $\mathbb{R}_{+}^{N}$ and $w(y_{0})=1$.
Then $w$ 
is a bounded solution of
 \begin{align*}
 \left\{\begin{array}{llll}
 (-\Delta)^{s} w&=&w & \text { in } \R^N_+,\\
\qquad w&\geq&0& \text { in } \R^N_+,\\
  \end{array}\right.
 \end{align*}
moreover, $w>0$ in $\R^N_+$ and this yields a contradiction, by Corollary \ref{liouv2}. This shows that $M_n^{p_n-1}$ is bounded and concludes the proof of {\sl Step 1}.

\medskip

{\sl Step 2. Conclusion.}\\
$M_n^{p_n-1}$ is bounded by {\sl Step 1.} Thus, up to a subsequence, $M_n^{p_n-1}\to \mu$. Let $z_n:=\frac{u_n}{M_n}$, then it satisfies that
\begin{align*}
\left\lbrace 
\begin{array}{ll}
(-\Delta)^{s}z_n= -\l z_n+ M_n^{p_n-1}\,z_n^{p_n}=:g_n &\text{   in } \O,\\
 0<z_n\leq 1 &\text{   in } \O, \\
z_n=0 &{\text{ in }}\ren\backslash \O.\\
\end{array}\right.
\end{align*}
By Lemma~\ref{threg}, 
\[\|z_n\|_{C^s(\overline\O)}\leq C  \|g_n\|_{L^{\infty}(\O)}\leq C. \]
Hence, $z_n$ converges to $z$ in $C(\overline\O)$ by Arzelà-Ascoli's Theorem. 
Moreover, similarly to the proof of \eqref{vnlocReg}, using Lemmas~\ref{threg}, \ref{lemma:interiorRegularity1}, and  \ref{lemma:interiorRegularity2}, we have that, for all $\beta \in (0,\min\{2s,1\})$ and for any compact set $K\subset\subset \O$, there exist $n_K\in\mathbb N$ and $C=C(s,N,\beta,K)>0$
such that
\begin{align}
\label{znlocReg}
z_n\in C^{2s+\beta}(\overline K) \ \mbox{ and }\ \|z_n\|_{C^{2s+\beta}(\overline K)}\leq C, \quad \forall n\geq n_K.\end{align}
Hence $z\in C^{2s+\frac{\beta}{2}}(\O)$ and $z_n$ converges to $z$ in $C_{loc}^{2s+\frac{\beta}{2}}(\Omega)$; furthermore, $0\leq z\leq 1$ and it  satisfies that
\begin{align*}
\left\lbrace 
\begin{array}{rcll}
(-\Delta)^{s}z&=&(-\l+\mu) z&\text{ in } \O,\\
z&\geq&0&\text{ in }\O,\\
z&=&0 &\text{ in }\ren\backslash \O.\\
\end{array}\right.
\end{align*}
Let $(x_n)_{n\geq 1}$ be a sequence in $\O$ such that $M_n=u_n(x_n)$, then  $1=\lim_{n\to \infty}z_n(x_n)=z(\bar{x})$ and then $\bar{x}\in \O$. Hence, $z>0$ and so $\mu=\l_1^s(\O)+\l$, $z=\varphi_1^s$,  and the Lemma is proved.
\end{proof}

\

\subsection{Proof of Theorem~\ref{exit p small}.} 

\begin{proof}[Proof of Theorem~\ref{exit p small}]
{\sl Step 1. We prove the nondegeneracy.}\\  Assume by contradiction that there exists a non trivial solution $h_n$ of the linearized problem \eqref{linearized} with $p=p_n>1$ and $p_n\to1$: 
\begin{align}
\label{linearizedproblemFonNondegeneracy}
\left\lbrace 
\begin{array}{rcll}
(-\Delta)^{s} h_n&= & - \l h_n + p_nu_n^{p_n-1}h_n\ =:g_n & {\text{ in }} \O,\\
h_n &= & 0 & {\text{ in }}\mathbb R^N\backslash \O, \\
\end{array}\right.
\end{align}
w.l.o.g. we may assume that $\|h_n\|_{L^\infty}=1$. Observe  that 
\[ 
\|g_n\|_{L^\infty}\leq |\lambda|+2\|u_n\|^{p_n-1}_{L^\infty}=|\lambda|+2M_n^{p_n-1}\leq C.
\]
So, by Lemma~\ref{threg},  $h_n\to h$  uniformly on $\Omega$; in particular $\|h\|_{L^\infty}=1$ and so $h\neq 0$.  
Furthermore, by  Lemma~\ref{threg}, $h_n  \in C^\infty(\Omega)\cap C^{s}(\R^N)\cap L^\infty(\R^N)$.
Taking $h_n$ as a test function, by Lemma~\ref{pclose1},  we derive that
\[\|h_n\|_s^2\leq\left(|\lambda|+2M_n^{p_n-1}\right)\|h_n\|_{L^2}^2\leq \left(|\lambda|+2M_n^{p_n-1}\right)\|h_n\|_{L^\infty}^2=|\lambda|+2M_n^{p_n-1}\leq C.\]
Hence, $h_n$ converges to $h$ also weakly in  $\cH^s_0(\Omega)$, up to a subsequence, and strongly in $L^2(\O)$. Passing to the limit, in the weak formulation of \eqref{linearizedproblemFonNondegeneracy}, 
we obtain that $h$ is a weak solution of
\begin{align*}
\left\lbrace 
\begin{array}{rcll}
(-\Delta)^{s} h&= &  \l_1^s(\O) h  & {\text{ in }} \O,\\
h &= & 0 & {\text{ in }}\RR\backslash \O, \\
\end{array}\right.
\end{align*}
since, by Lemma~\ref{pclose1}, one has that
\begin{equation}\label{limitOfunp}
p_nu_n^{p_n-1}=p_nM_n^{p_n-1}\left(\frac{u_n}{M_n}\right)^{p_n-1}=\lambda^s_1(\O)+\lambda+o(1)\quad \text{pointwisely in $\O$ as }n\to\infty\end{equation}
and
\begin{align}\label{limit2}
\|p_nu_n^{p_n-1}\|_{L^\infty}\leq \lambda^s_1(\O)+|\lambda|+1\qquad \text{ for all }n\in\N.
\end{align}
As a consequence, we may assume that $h=\varphi_1^s$, where $\varphi_1^s$ is the first Dirichlet eigenfunction of $(-\Delta)^s$ in $\Omega$. Note that $h_n$ must change sign, because
\begin{equation}\label{segno} 0=\int_\O h_n[(-\Delta)^{s}u_n+\l u_n]-u_n[(-\Delta)^{s}h_n+\l h_n]\, dx=(1-p_n)\int_\O u_n^{p_n}h_n\, dx.\end{equation}
By  \eqref{segno}, Lemma~\ref{pclose1}, and dominated convergence, we have that
\[0=\frac{1}{M_n^{p_n}}\int_\O u_n^{p_n}h_n\, dx=\int_{\O} \left(\frac{u_n}{M_n}\right)^{p_n-1}\frac{u_n}{M_n}h_n=\int_{\O}(\varphi_1^s)^2+o(1)\qquad \text{ as }n\to\infty,\]
which leads to a contradiction.

\medskip

{\sl Step 2. We prove the uniqueness.} \\
By contradiction, assume that $u_n$ and $v_n$ are two distinct solutions of problem \eqref{main} with $p=p_n>1$ and $p_n\to1$.  The functions $w_n:= \frac{u_n-v_n}{\|u_n-v_n\|_{L^\infty}}$ satisfy that
\begin{align}\label{w_neq}
\left\{
\begin{array}{ll}
(-\Delta)^{s}w_n=\alpha_n(x)w_n-\l w_n=:g_n &\text{ in } \O,\\
w_n=0&\text{ in }\ren\backslash \O,
\end{array}
\right.
\end{align}
where, by the Mean Value Theorem,
\begin{align*}
\alpha_n:=\int_0^1 p_n(t u_n +(1-t) v_n)^{p_n-1}\, dt.
\end{align*}
Since $t u_n(x) +(1-t) v_n(x)$ is between $u_n(x)$ and $v_n(x)$,  it follows that $\alpha_n(x)$ is between $p_nu_n(x)^{p_n-1}$ and $p_nv_n(x)^{p_n-1}$ for all $x\in\O$. Since, by \eqref{limitOfunp} and \eqref{limit2}, $\|\alpha_n\|_{L^\infty}\leq \l_1^s(\O)+|\lambda|+1$ for all $n\in\N$  and 
$p_nu_n^{p_n-1}\to \l_1^s(\O)+\l$, $p_nv_n^{p_n-1}\to \l_1^s(\O)+\l$ pointwisely in $\O$ as $n\to\infty$,
we have that
\begin{align*}
\alpha_n(x)\rightarrow \l_1^s(\O)+\l\text{ pointwisely in $\Omega$ as } n\to\infty.    
\end{align*}
So $\|g_n\|_{L^\infty}\leq C$, hence by Lemma~\ref{threg}, $w_n\to w$ uniformly in $\overline{\O}$;
 in particular $\|w\|_{L^\infty}=1$ and so $w\neq 0$.
 Furthermore, testing \eqref{w_neq} with $w_n$,
\[\|w_n\|_s^2\leq \left(|\lambda|+\|\alpha_n\|_{L^\infty}\right)\|w_n\|_{L^2}^2\leq \left(|\lambda|+\|\alpha_n\|_{L^\infty}\right)\|w_n\|_{L^\infty}^2=|\lambda|+\|\alpha_n\|_{L^\infty}\leq C.\]
Hence, $w_n$ converges to $w$ also weakly in  $\cH^s_0(\Omega)$, up to a subsequence, and strongly in $L^2(\O)$.

Passing to the limit in the weak formulation of \eqref{w_neq} we obtain that $w$ is a weak solution of
\begin{align*}
\left\lbrace 
\begin{array}{rcll}
(-\Delta)^{s}w&=&\l_1^s(\O)w&\text{ in } \O,\\
w&\neq&0&\text{ in } \O,\\
w&=&0 &\text{ in }\ren\backslash\O.
\end{array}\right.
\end{align*}
Hence, $w=\varphi_1^s$ the first eigenfunction associated to $\l_1^s(\O)$. Let us observe that also 
\begin{align}\label{contra22}
\frac{w_n}{\d^s}\to \frac{\varphi_1^s}{\d^s}\, \text{ in } C^{\beta}(\overline{\O}),
\end{align}
for some $\beta\in (0,s)$ . Indeed, by Lemma~\ref{threg}, one also has that
$
\frac{w_n}{\d^s}\to \xi \text{  in } C^{\beta}(\overline{\O}),
$ for $\beta\in (0,\alpha)$,
and it is easy to see that $\xi=\frac{\varphi_1^s}{\delta^s}$, because the uniform convergence of $w_n$ to $\varphi_1^s$ implies that $\xi|_{\Omega}\equiv\frac{\varphi_1^s}{\delta^s}|_{\Omega}$, and  both $\xi|_{\Omega}$ and $\frac{\varphi_1^s}{\delta^s}|_{\Omega}$ can be uniquely extended in $\overline\O$.

Furthermore, we show that $w_n$ must change sign in  $\O$. Otherwise, if $w_n\geq 0$, namely $u_n\geq v_n$, then  $u_n^{p_n-1}- v_n^{p_n-1}\geq 0$ (since $p_n>1$), and so   the  equality 
\[0=\int_\O v_n[(-\Delta)^{s}u_n+\l u_n]-u_n[(-\Delta)^{s}v_n+\l v_n]\, dx
=\int_{\O}v_nu_n \left(u_n^{p_n-1}-v_n^{p_n-1}\right)\, dx \]
would imply  $u_n^{p_n-1}\equiv v_n^{p_n-1}$, namely,  $u_n\equiv v_n$, a contradiction.

Let now $(x_n)_{n\geq 1}$ be a sequence in $\O$ such that $w_n(x_n)=\min_{x\in\overline{\O}}w_n(x)$. Then, since $w_n$ changes sign in $\Omega$ and $w_n\to\varphi_1^s>0$ uniformly in $\overline\O$, one has that
\begin{align*}
 w_n(x_n)< 0 \quad \text{  and  }\quad  x_n\to x^* \in\partial\O.
\end{align*}
Using \eqref{contra22}, $\lim_{n\to\infty}\frac{w_n}{\d^s}(x_n)=\frac{\varphi_1^s}{\d^s}(x^*)$; hence,
$\frac{\varphi_1^s}{\d^s}(x^*)\leq 0,$ which gives a contradiction, since $\frac{\varphi_1^s}{\d^s}> 0$ on $\partial\O$, by the fractional Hopf Lemma (see \cite{GS16, FJ15}). 
\end{proof}


\section{The asymptotically local case} \label{sec:l}

Recall that $\lambda_1^s(\Omega)>0$ denotes the first Dirichlet eigenvalue of $(-\Delta)^s$ in $\Omega$ for $s\in(0,1]$.  In this section we fix 
\begin{align*}
p\in(1,2^*-1),\qquad \lambda>-\lambda_1^1(\Omega),
\end{align*}
and $\Omega\subset \R^N$ such that the problem
\begin{align}\label{c}
    -\Delta u +\lambda u&= u^{p}\quad \text{ in }\Omega,\qquad u=0\quad \text{ on }\partial \Omega,
\end{align}
has a unique positive solution and it is nondegenerate, namely the linearized problem
\begin{align}\label{Omega}
    -\Delta v +\lambda v&= pu^{p-1}v,\qquad v\in H^1_0(\Omega),\qquad \text{only admits the trivial solution.}
\end{align}

Since $\lim_{s\to 1}\lambda_1^s(\Omega)=\lambda_1^1(\Omega)$ (see Remark \ref{rmk:e}), there is $\sigma_0\in(\frac{1}{2},1)$ such that, for all $s\in [\sigma_0,1]$,
\begin{align}\label{s}
p<\frac{N+2s}{N-2s}\qquad \text{ and }\qquad  \lambda>-\lambda_1^s(\Omega).
\end{align}

For $s\in [\sigma_0,1]$, let ${\mathcal M}_s$ denote the set of positive nontrivial solutions $u_s$ of 
\begin{align}\label{le:eq}
    (-\Delta)^s u_s +\lambda u_s= u_s^p\quad \text{ in }\Omega,\qquad u_s=0\quad \text{ in }\R^N\backslash \O.
\end{align}

The following result gives a uniform a priori bound for all positive solutions whenever $s$ is close to 1.  The proof follows a blow-up argument similar to that of Lemma~\ref{pclose1}; however, since now the blow-up parameter is the Laplacian's exponent $s$, special care is required to control the constants appearing from regularity estimates.  In particular, a priori $C_{loc}^{2s+\eps}$  regularity estimates with explicit constants do not seem to be available in the literature and are nontrivial (see Remark \ref{rmk:c} below).  To overcome this obstacle, we use the lower order $C^s$ estimates with explicit constants shown in \cite{JSW20} and use the regularity theory for distributional solutions. 

\begin{theorem}\label{thm:unif}
Let $N\geq 2$, $\Omega$ be a bounded domain of class $C^2$, $1<p<2^*-1$, $\lambda>-\lambda_1^1(\Omega)$. There is $\sigma\in(0,1)$ and a constant $C=C(\lambda,p,\Omega,\sigma)>0$ such that,
\begin{align*}
\|u_s\|_{L^\infty}<C\qquad \text{for all $s\in(\sigma,1)$ and $u_s\in {\mathcal M}_s$.}
\end{align*}
\end{theorem}
\begin{proof}
We argue by contradiction.  Let $(s_n)\subset (\frac{1}{2},1)$ be such that $\lim_{n\to\infty}s_n=1$, $p\in(1,2^*-1)$, $\lambda>-\lambda_1^1(\Omega)$, and assume that there are positive solutions $u_n\in C^\infty(\Omega)\cap C^{s}(\R^N)\cap L^\infty(\R^N)$ of \eqref{le:eq} with $s=s_n$ and such that $M_n:=\|u_n\|_{L^\infty}\to\infty$.  Let $x_n$ be the points where the maximum of $u_n$ is achieved, $\mu_n:=M_n^\frac{1-p}{2s_n}$, and let $
    v_n(y):=\frac{1}{M_n}u_n(\mu_n y+x_n)$ and $\Omega_n:=\{y\in \Omega\::\: \mu_n x+x_n\in \Omega\}.$ Then $v_n$ is a function satisfying that $0 \leq v_n\leq 1$, $v_n(0) = 1$, and
\begin{align}
\label{v_n eqs:2}
\begin{array}{lr}
(-\Delta)^{s_n} v_n=v_n^{p}-\frac{\l}{M_n^{p-1}}v_n=:f_n\quad \text{ in } \O_n,\qquad v_n=0 \text{ in }\ren\backslash \O_n.
 \end{array}
\end{align}
Passing to a subsequence, we consider two cases. 

\medskip

\underline{Case 1: $\lim_{n\to\infty}\operatorname{dist}\left(x_{n},\,\partial\O\right) \mu_{n}^{-1} =\infty$}.
Then $\Omega_n \rightarrow \ren$ as $n \rightarrow+\infty$ and, by Lemma~\ref{lem:int:reg}, for any $r>1$ with
$B_r(0)\subset \Omega_n$, there is $C=C(N)>0$ such that
\begin{align*}
&[v_n]_{C^{s_n}(B_{\frac{r}{2}})}\\
&\leq r^{s_n} C
\left(
\left\|v_n^p - \frac{\lambda}{M_n^{p-1}} v_n\right\|_{L^\infty(B_r)}
+\frac{2}{\Gamma(s_n)\Gamma(1-s_n)|\mathbb S^{N-1}|}
\int_{\R^N\backslash B_r}\frac{|v_n(z)|}{|z|^N(|z|^2-r^2)^{s_n}}\, dz
\right)\\
&\leq 2(1+|\lambda|)r^{s_n} C
\left(1+\frac{1}{\Gamma(s_n)\Gamma(1-s_n)}
\int_{r}^\infty\frac{1}{\rho(\rho^2-r^2)^{s_n}}\, d\rho
\right)= 2(1+|\lambda|)r^{s_n} C
\left(1+\frac{1}{2r^{2 s_n}}
\right),
\end{align*}
were we used that, by the properties of the beta function $B(\cdot,\cdot)$,
\begin{align*}
    \int_r^\infty \rho^{-1} (\rho^2-r^2)^{-s_n} \, d\rho
 &=r^{-2s_n}\int_1^\infty t^{-1} (t^2-1)^{-s_n} \, dt=\frac{r^{-2s_n}}{2}\int_1^\infty (\tau-1)^{-s_n}\tau^{-1}\, d\tau\\
    &=\frac{r^{-2s_n}}{2}\int_0^1 (\zeta^{-1}-1)^{-s_n}\zeta^{-1}\, d\zeta=\frac{r^{-2s_n}}{2}\int_0^1 (1-\zeta)^{-s_n}\zeta^{s_n-1}\, d\zeta\\
    &=\frac{r^{-2s_n}}{2}B(s_n,1-s_n)=\frac{r^{-2s_n}}{2}\frac{\Gamma(s_n)\Gamma(1-s_n)}{\Gamma(1)}.
\end{align*}

Using that $s_n$ is an increasing sequence, we deduce that
\begin{align*}
[v_n]_{C^{s_1}(B_{\frac{r}{2}})}
\leq r^{s_1-s_n}[v_n]_{C^{s_n}(B_{\frac{r}{2}})}
\leq 4(1+|\lambda|)C r^{2}.
\end{align*}
This uniform bound and a diagonalization argument yields the existence of $v\in C^{\beta}_{loc}(\R^N)$ such that 
$v_n\to v$ in $C^{\beta}_{loc}(\R^N)$ for any $\beta\in(0,1)$.  Moreover, by \eqref{v_n eqs:2}, $v$ satisfies that $v(0)=1$, $v\geq 0$ in $\R^N$, and 
\begin{align*}
    \int_{\R^N} v(-\Delta)\varphi\, dx=\int_\Omega v^{p} \varphi\, dx
    \qquad \text{for all $\varphi\in C^\infty_c(\R^N)$},
\end{align*}
namely, $v\in L^\infty_{loc}(\R^N)$ is a distributional solution of
\begin{align}\label{l}
-\Delta v=v^{p}\quad \text{ in $\R^N$.}
\end{align}
By the regularity theory for distributional solutions in $L^2_{loc}(\R^N)$ (see, for example, \cite[Theorem 1]{AF20}), we have that $v\in W^{2,2}_{loc}(\R^N)$ solves \eqref{l} weakly, and since $v\in C^\beta_{loc}(\R^N)$, standard elliptic regularity yields that $v$ is a classical solution of \eqref{l} which is positive by the maximum principle. But this contradicts the classical Liouville theorem in the whole space (see \cite[Theorem 1.1]{GidasSpruckCPAM1981}). 

\medskip

\underline{Case 2: $\operatorname{dist}\left(x_{n},\,\partial\O\right) \mu_{n}^{-1} \rightarrow d \geq 0$}. We may assume $x_{n} \rightarrow x_{0} \in \partial \Omega .$ With no loss of generality assume also $\nu\left(x_{0}\right)=-e_{N}$. In this case, we use the function
$w_{n}(y)=u_{n}\left(\mu_{n} y+\xi_{n}\right)M_n^{-1}$ for $y \in D_{n}$, where $\xi_{n} \in \partial \Omega$ is such that $|x_{n}-\xi_n|=\operatorname{dist}(x_n,\partial \Omega)$ and $
D_{n}:=\left\{y \in \mathbb{R}^{N}: \mu_{n} y+\xi_{n} \in \Omega\right\}$. Then $w_{n}$ satisfies that
\begin{align}
\label{v_n eq2:2}
(-\Delta)^{s_n} w_n=w_n^{p}-\frac{\l}{M_n^{p-1}}w_n \quad \text{ in } D_n,\qquad
 w_n=0 \text{ in }\ren\backslash D_n.
\end{align}
Moreover, setting $y_{n}:=(x_{n}-\xi_{n})\mu_{n}^{-1}$, it follows that $\left|y_{n}\right|=\operatorname{dist}\left(x_{n},\,\partial\O\right) \mu_{n}^{-1}$ and $w_{n}\left(y_{n}\right)=1$. Since $\Omega$ is of class $C^2$, it satisfies a uniform exterior sphere condition, namely, there is $r_0>0$ such that, for every $x_0\in \partial \Omega$, there is $y_0\in \R^N\backslash \Omega$ with $\overline{B_{r_0}(y_0)}\cap \overline{\Omega} = \{x_0\}$. Furthermore, since $\mu_n^{-1}\to \infty$, $D_n$ also satisfies the exterior sphere condition with the same $r_0$ for every $n\in\N$. Then, by Lemma~\ref{lem:unif} and Remark \ref{rmk:bdry} (with that $r_0$ independent of $n$, $\sigma=\frac{1}{2},$ and $M=3$) there is $\delta_0>0$ independent of $n$ such that $w_n(y)<\frac{1}{2}$ for all $y\in D_n$ with $\operatorname{dist}\left(y,\,\partial D_n\right)<\delta_0$. Since $w_n(y_n)=1$, this implies that $\operatorname{dist}\left(y_n,\,\partial D_n\right)\geq \delta_0$, namely, 
$\rho=\operatorname{dist}\left(x_{n},\,\partial\O\right) \mu_{n}^{-1}>\delta_0$.  In particular, passing to a subsequence, $y_{n} \rightarrow y_{0}$, where $\left|y_{0}\right|=\rho>0,$ thus $y_{0}$ is an interior point of the half-space $\R^N_+$.  Finally, arguing similarly as in the first case, we obtain that $w_{n} \rightarrow w$ in $C^{\beta}_{loc}(\ren_{+})$ for all $\beta\in (0,1)$, where $w$ is a classical positive solution of $-\Delta w=w^{p}$ in $\R^N_+$.  Moreover, by Lemma~\ref{lem:unif}, $w\in C(\overline{\R^N_+})$ and $w=0$ on $\partial \R^N_+$; but this contradicts the classical Liouville theorem in the halfspace (see \cite[Theorem 1.3]{GidasSpruckCPDE1981}).
\end{proof}

\begin{remark}\label{rmk:c}
The limit as $s\to 1^-$ is delicate, because it is the transition between the nonlocal and the local regime. In the proof of Theorem~\ref{thm:unif}, we only use uniform (in $s$) lower order regularity estimates. This yields distributional solutions of \eqref{l} in the limit, which, using local arguments, can then be shown to be regular. Comparing this argument with the proof of Lemma~\ref{pclose1}, one may think that using uniform higher order regularity estimates would be simpler to obtain directly a classical solution of \eqref{l}. These precise higher order regularity estimates are not yet available in the literature and require a careful analysis of the constants involved in the known regularity arguments for the fractional Laplacian, where the explicit dependence on $s$ is often disregarded. We point out that this can be subtle issue, since several constants and integrals involved in the analysis of fractional problems have a singular behavior in the nonlocal-to-local transition, namely, when $s\to 1^-$.  For instance, for $N=2$ and $s\in(0,1)$, the fundamental solution for $(-\Delta)^s$ in $\R^2$ is given by  
\begin{align}\label{c:def}
 F_{2,s}(x):=a_{s} |x|^{2s-2}\quad \text{ for $x\in \R^2\setminus\{0\}$},\qquad a_{s}:=\frac{\Gamma(1-s)}{4^{s}\pi\Gamma(s)},
\end{align}
and the constant $a_{s}$ blows up as $s\to 1^-$. As a consequence, regularity estimates that use the fundamental solution (see \emph{e.g.} \cite[Proposition 2.8]{silver}) need to be refined to obtain uniform constants in the limit as $s\to 1^-$.
\end{remark}

The following Lemma is one of our main asymptotic tools. It exploits the uniform regularity estimates given in Lemma \ref{threg} and describes the properties of the limiting profile.

\begin{Lemma}\label{Fatou:arg}
Let $\Omega\subset \R^N$ be a bounded domain of class $C^2$, $(s_n)\subset(0,1)$, $\lim_{n\to\infty}s_n=1$, and let $(u_n)\subset C^\infty(\Omega)\cap \cH^{s_n}_0(\Omega)$ be a sequence such that
\begin{align}\label{C:def}
\|u_n\|_{s_n}+\|(-\Delta)^{s_n}u_n\|_{L^\infty(\Omega)}<C\qquad \text{ for all $n\in \N$ and for some $C>0$.}
\end{align} Then, passing to a subsequence, there is $u^*\in H_0^1(\Omega)\cap C^\beta(\overline{\Omega})$ for all $\beta\in(0,1)$ such that 
$u_n\to u^*$ in $C^\beta(\Omega)$ as $n\to \infty$
and 
\begin{align*}
\int_\Omega \nabla u^* \nabla \varphi\, dx 
=\lim_{n\to\infty}\int_\Omega (-\Delta)^{s_n} u_n\, \varphi\, dx\qquad \text{ for all }\varphi\in C^\infty_c(\Omega).
\end{align*}
\end{Lemma}
\begin{proof}
By \eqref{C:def} and Lemma~\ref{threg}, there is $u^*:\Omega\to\R$ such that, given $\beta\in(0,1)$ and passing to a subsequence, $u_{n} \to u^*$ in $C^\beta(\Omega)$ (in particular, $u^*=0$ on $\partial \Omega$). Moreover, $u^*\in H^1_0(\Omega)$ since, by Fatou's Lemma, 
 \begin{align*}
\|u^*\|_{1}^2
=\int_{\R^N}|\xi|^{2}|\widehat u^*(\xi)|^2\, d\xi
\leq \liminf_{n\to\infty}\int_{\R^N}|\xi|^{2s_n}|\widehat u_{s_n}(\xi)|^2\, d\xi
= \liminf_{n\to\infty}\|u_{s_n}\|_{s_n}^2
<C.
 \end{align*}
 Then, for every $\varphi\in C^\infty_c(\Omega)$, integrating by parts,
 \begin{align*}
 \int_\Omega \nabla u^* \nabla \varphi\, dx
 &=\int_\Omega u^* (-\Delta)\varphi\, dx
 =\lim_{n\to\infty}\int_\Omega u^* (-\Delta)^{s_n}\varphi\, dx\\
& =\lim_{n\to\infty}\int_\Omega u_n (-\Delta)^{s_n} \varphi\, dx
 =\lim_{n\to\infty}\int_\Omega (-\Delta)^{s_n} u_n \varphi\, dx,
 \end{align*}
 where we used that $(-\Delta)^{s_n}\varphi \to (-\Delta)\varphi$ pointwisely as $n\to\infty$ and that $\|(-\Delta)^{s_n}\varphi\|_{L^\infty}$ is uniformly bounded independently of $n$, see for example \cite[Proposition 4.4]{DiNezza} and \cite[Lemma~B.5]{AJS18}.
\end{proof}

The following result is known for ground states, see, for instance, \cite[Theorem 4.5.]{BS20}, \cite[Theorem 1.2]{BS21}, or \cite[Theorem 1.1]{HS21} for critical equations (the argument can be easily adapted to subcritical problems).  We extend these results to general positive solutions. 

\begin{Lemma}\label{ustar:lem}
Let $1<p<2^*-1$, $\lambda>-\lambda_1^1(\Omega)$, $(s_n)\subset (0,1)$ be such that $\lim_{n\to\infty}s_n=1$, and let $u_{n}\in {\mathcal M}_{s_n}.$  Then, up to a subsequence, there is $u_*\in {\mathcal M}_1$ such that $u_{n}\to u_*$ in $L^{p+1}(\Omega)$.
\end{Lemma}
\begin{proof}
By Theorem~\ref{thm:unif}, passing to a subsequence, \eqref{s} holds for $s=s_n$ and there is $C=(\lambda,p,\Omega)>0$ such that 
\begin{align*}
\|u_{n}\|_{s_n}^2=\|u_{n}\|_{L^{p+1}}^{p+1}-\lambda \|u_{n}\|_{L^2}^2\leq |\Omega|(\|u_{n}\|^{p+1}_{L^\infty}+|\lambda|\|u_{n}\|_{L^\infty}^2)<C\qquad \text{ for all }n\in\N.
\end{align*}
By Lemma \ref{Fatou:arg}, there is $u^*\in H^1_0(\Omega)$ such that $u_n\to u^*$ in $L^\infty(\Omega)$ and, for every $\varphi\in C^\infty_c(\Omega)$,
 \begin{align*}
 \int_\Omega \nabla u^* \nabla \varphi\, dx
  =\lim_{n\to\infty}\int_\Omega (u_n^p-\lambda u_n) \varphi\, dx
 =\int_\Omega ((u^*)^p-\lambda u^*) \varphi\, dx.
 \end{align*}
 It remains to show that $u^*\not\equiv 0$. Let $v_{n}\in {\mathcal M}_{s_n}$ denote the \emph{least energy solution} of \eqref{main}; then, by \cite[Theorem 4.5.]{BS20}, (see also \cite[Theorem 1.2]{BS21}), we have that $v_n\to v^*$ in $L^{p+1}(\Omega)$ to some $v^*\in {\mathcal M}_1$, and therefore,
 \begin{align*}
 \|u^*\|_{L^{p+1}}^{p+1}
     &=\lim_{n\to \infty}
     \|u_{n}\|_{L^{p+1}}^{p+1}
     \geq \lim_{n\to\infty}
     \|v_{n}\|_{L^{p+1}}^{p+1}
     =
     \|v^*\|_{L^{p+1}}^{p+1}>0.
 \end{align*}
 Then $u_*\not\equiv 0$ and therefore $u_*\in {\mathcal M}_1$.
\end{proof}

\begin{remark}\label{rmk:e}
It is well known that $\lim_{s\to 1}\lambda^s_1(\Omega)=\lambda_1^1(\Omega)$, but we could not find a precise reference. In this Remark, we give a brief argument. Let $\varphi_1^s\in C^\infty(\Omega)\cap C^s(\R^N)$ denote the first eigenfunction associated to $\lambda^s_1(\Omega)$ for $s\in(0,1]$ normalized such that $\|\varphi_1^s\|_{L^\infty}=1$. Then, since $\varphi^1_1\in \cH^s_0(\Omega)$,
\begin{align*}
\lim_{s\to 1}\lambda_1^s(\Omega) = \lim_{s\to 1}\inf_{v\in \cH^s_0(\Omega)}\frac{\|v\|_s^2}{\|v\|^2_{L^2}}\leq \lim_{s\to 1}\frac{\|\varphi^1_1\|_s^2}{\|\varphi^1_1\|^2_{L^2}}
=\frac{\|\varphi^1_1\|_1^2}{\|\varphi^1_1\|^2_{L^2}}=\lambda_1^1(\Omega).
\end{align*}
In particular, $(\lambda^s_1(\Omega))$ is bounded and therefore $(\|\varphi_1^s\|_s+\|(-\Delta)^s\varphi_1^s\|_{L^\infty(\Omega)})$ is bounded as $s\to 1^-$ as $n\to\infty$. Let $s_n\to 1^-$. Then by Lemma \ref{Fatou:arg}, passing to a subsequence, $\varphi_1^{s_n}\to \varphi$ in $L^\infty(\Omega)$ as $n\to\infty$ for some $\varphi\in H^1_0(\Omega)$ and, by Fatou's Lemma,
\begin{align*}
\lambda_1^1 (\Omega)=\inf_{v\in H^1_0(\Omega)}\frac{\|v\|_1^2}{\|v\|^2_{L^2}}\leq \frac{\|\varphi\|_1^2}{\|\varphi\|^2_{L^2}}
\leq \liminf_{n\to \infty}\frac{\|\varphi^{s_n}_1\|_{s_n}^2}{\|\varphi^{s_n}_1\|_{L^2}^2}=\liminf_{n\to \infty}\lambda_1^{s_n}(\Omega).
\end{align*}
Since this can be done for any subsequence, we conclude that $\lim_{s\to 1}\lambda^s_1(\Omega)=\lambda_1^1(\Omega)$.
\end{remark}

\subsection{Proof of Theorem~\ref{thm:scto1}}

 \begin{proof}[Proof of Theorem~\ref{thm:scto1}]
 Let $N\geq 2$, $p\in(1,2^*-1)$, $\lambda>-\lambda_1^1(\Omega)$, and $\Omega$ be such that the problem \eqref{c:intro} has a unique positive solution which is nondegenerate. Then, for $s$ sufficiently close to 1, \eqref{s} holds. We argue first the nondegeneracy of solutions of \eqref{main}. By contradiction, let $(s_n)\subset (0,1)$ be a monotone increasing sequence with $\lim_{n\to \infty}s_n=1$, $u_{n}\in \mathcal{M}_{s_n}$,  and assume that there is a solution $h_{n}$ of 
\begin{align*}
    (-\Delta)^{s_n} h_{n} = pu_{n}^{p-1}h_{n}-\lambda h_{n},\quad h_{n}\in \cH^{s_n}_0(\Omega)\cap C^\infty(\Omega),\quad \|h_{n}\|_{L^\infty}=1\qquad \text{for all $n\in\N.$  }
\end{align*}
In the following, we use $C>0$ to denote possibly different constants independent of $n$. Observe that, by Theorem~\ref{thm:unif}, passing to a subsequence,
\begin{align*}
\|h_{n}\|_{s_n}^2
=p\int_{\Omega}u_{n}^{p-1}h_{n}^2-\lambda \|h_{n}\|_{L^2}^2
\leq |\Omega|(p\|u_{n}\|_{L^\infty}^{p-1}+|\lambda|)\leq C\qquad \text{for all $n\in\N.$}
\end{align*}

By Lemmas \ref{Fatou:arg} and \ref{ustar:lem}, there are $v\in H^1_0(\Omega)$ and $u\in {\mathcal M}_1$ such that $h_{n}\to h$ and $u_{n}\to u$ in $L^{\infty}(\Omega)$. Then,
\begin{align*}
\lim_{n\to\infty}\int_\Omega u_{n}^{p-1}h_{n}\varphi\, dx=\int_\Omega u^{p-1}h\varphi\, dx.
\end{align*}
Note that $u$ is the unique solution of \eqref{c}. Then, by Lemma \ref{Fatou:arg},
\begin{align*}
\int_{\Omega} \nabla h\nabla\varphi\, dx
&=\lim_{n\to\infty}\int_\Omega  (pu_{n}^{p-1}-\lambda)h_{n}\varphi\, dx
=\int_\Omega  (pu^{p-1}-\lambda)h\varphi\, dx,
\end{align*}
for $\varphi\in C^\infty_c(\Omega)$, which contradicts the nondegeneracy of the limiting problem \eqref{Omega}.

Next we prove the uniqueness of solutions of \eqref{main} for $s$ sufficiently close to 1.  We argue as in Theorem~\ref{exit p small}.  By contradiction, assume that $u_n$ and $v_n$ are two distinct solutions of problem \eqref{main} with $p\in(1,2^*-1)$ and $s_n\to 1^-$.  By Lemma~\ref{ustar:lem} and the uniqueness of solutions of the problem with $s=1$, we have that $u_n\to u$, $v_n\to u,$ in $L^{\infty}(\Omega),$ where $u\in {\mathcal M}_1$ is the unique solution of \eqref{main} for $s=1$. Let $w_n:= \frac{u_n-v_n}{\|u_n-v_n\|_{L^\infty}}$, then
 \begin{align*}
 (-\Delta)^{s}w_n=\alpha_n(x)w_n-\l w_n=:g_n \quad\text{ in } \O, \qquad 
w_n=0\quad \text{ in }\ren\backslash \O,\qquad \|w_n\|_{L^\infty}=1,
 \end{align*}
 where $\alpha_n:=\int_0^1 p(t u_n +(1-t) v_n)^{p-1}\, dt$ satisfies that $\|\alpha_n\|_{L^\infty} <C$ for all $n\in\N$ and for some $C>0$ (by Theorem~\ref{thm:unif}) and, by dominated convergence, $\alpha_n\to p u^{p-1}$ a.e. in $\Omega$ as $n\to\infty.$   By Lemma~\ref{Fatou:arg}, there is $w\in H_0^1(\Omega)\backslash \{0\}$ such that $w_n\to w$ in $L^\infty(\Omega)$ and, again by dominated convergence,
 \begin{align*}
\int_\Omega \nabla w \nabla \varphi\, dx
=\lim_{n\to\infty}\int_\Omega(\alpha_n(x)w_n-\l w_n) \varphi\, dx
=\int_\Omega (pu^{p-1}-\lambda)w\varphi\, dx
 \end{align*}
 for all $\varphi\in C^\infty_c(\Omega)$, which contradicts the nondegeneracy of the limiting problem \eqref{Omega}.
\end{proof}

\begin{remark}[Uniform constants depending on upper bounds]\label{rmk:p:dep}
In the proof of Theorem~\ref{thm:scto1}, the constant $\sigma$ depends on $p$.  One can obtain a uniform $\sigma$ for all $p\in(1,p_0)$ with $p_0<2^*-1$ by considering also a sequence $p_n\subset (1,p_0]$ such that $\lim_{n\to\infty}p_n=\overline{p}\in[1,p_0]$ in the proof by contradiction of Theorem~\ref{thm:scto1}.  If $\overline{p}>1$, basically the same argument applies, whereas if $\overline{p}=1$, then one can argue as in Theorem~\ref{exit p small}.  Note, however, that this requires more technicalities, since uniform interior $C^{2s+\eps}$ regularity estimates are not available for $s\to 1^-$ (see Remark \ref{rmk:c}).  Instead, one can use regularity theory for distributional solutions, as in Theorem~\ref{thm:unif}. In order to make the ideas in our arguments more transparent, we do not pursue this here.
\end{remark}

\subsection{Proofs of the corollaries}

All the corollaries stated in the introduction follow directly from Theorem \ref{thm:scto1} and the known results for the local case, except for Corollary \ref{c:le}, which involves only least-energy solutions and it is a consequence of the next result.

\begin{theorem}
Let $N\geq 2$, $p\in(1,2^*-1)$, $\lambda>-\lambda_1^1(\Omega)$, and $\Omega$ be such that the problem 
\begin{align*}
    -\Delta u +\lambda u&= u^{p}\quad \text{ in }\Omega,\qquad u=0\quad \text{ on }\partial \Omega,
\end{align*}
has a unique least energy solution $u$ which is nondegenerate. Then, there is $\sigma=\sigma(\Omega,\lambda,p)\in(0,1)$ such that, for $s\in(\sigma, 1]$, the problem \eqref{main} has a unique least energy solution and it is nondegenerate.
\end{theorem}
\begin{proof}
The result follows by arguing exactly as in Theorem \ref{thm:scto1} and by noting that, by \cite[Theorem 4.5.]{BS20}, a sequence of least energy solutions converges to a least energy solution of the limiting problem.
\end{proof}

\begin{flushleft}
\textbf{Abdelrazek Dieb}\\
Department of Mathematics, Faculy of Mathematics and computer science \\
University Ibn Khaldoun of Tiaret\\
Tiaret 14000, Algeria\\
and\\
Laboratoire d’Analyse Nonlinéaire et Mathématiques Appliquées\\
Université Abou Bakr Belkaïd,Tlemcen, Algeria\\
\texttt{Abdelrazek.dieb@univ-tiaret.dz}
\vspace{.3cm}

\textbf{Isabella Ianni}\\
Dipartimento di Scienze di Base e Applicate per l’Ingegneria\\
Sapienza Universita di Roma\\
Via Scarpa 16, 00161 Roma, Italy\\
\texttt{isabella.ianni@uniroma1.it} 
\vspace{.3cm}

\textbf{Alberto Saldaña}\\
Instituto de Matemáticas\\
Universidad Nacional Autónoma de México\\
Circuito Exterior, Ciudad Universitaria\\
04510 Coyoacán, Ciudad de México, Mexico\\
\texttt{alberto.saldana@im.unam.mx} 
\vspace{.3cm}
\end{flushleft}


\begin{thebibliography}{10}

\bibitem{AJS18}
N.~Abatangelo, S.~Jarohs, and A.~Salda{\~n}a.
\newblock Green function and martin kernel for higher-order fractional
  laplacians in balls.
\newblock {\em Nonlinear Analysis}, 175:173--190, 2018.

\bibitem{abatangelo2015large}
Nicola Abatangelo.
\newblock Large $ s $-harmonic functions and boundary blow-up solutions for the
  fractional laplacian.
\newblock {\em Discrete \& Continuous Dynamical Systems}, 35(12):5555, 2015.

\bibitem{AdimurthiYadavaARMA1994}
Adimurthi and S.~L. Yadava.
\newblock An elementary proof of the uniqueness of positive radial solutions of
  a quasilinear {D}irichlet problem.
\newblock {\em Arch. Rational Mech. Anal.}, 127(3):219--229, 1994.

\bibitem{AftalionPacellaJDE2003}
A.~Aftalion and F.~Pacella.
\newblock Uniqueness and nondegeneracy for some nonlinear elliptic problems in
  a ball.
\newblock {\em J. Differential Equations}, 195(2):380--397, 2003.

\bibitem{BarriosColoradoServadeiSoriaHIHP2015}
B.~Barrios, E.~Colorado, R.~Servadei, and F.~Soria.
\newblock A critical fractional equation with concave-convex power
  nonlinearities.
\newblock {\em Ann. Inst. H. Poincar\'{e} C Anal. Non Lin\'{e}aire},
  32(4):875--900, 2015.

\bibitem{BartschClappGrossiPacellaMA2012}
T.~Bartsch, M.~Clapp, M.~Grossi, and F.~Pacella.
\newblock Asymptotically radial solutions in expanding annular domains.
\newblock {\em Math. Ann.}, 352(2):485--515, 2012.

\bibitem{BS21}
B.~Bieganowski and S.~Secchi.
\newblock Non-local to local transition for ground states of fractional
  schr{\"o}dinger equations on bounded domains.
\newblock {\em Topological Methods in Nonlinear Analysis}, 57(2):413--425,
  2021.

\bibitem{MR4062980}
Krzysztof Bogdan, Sven Jarohs, and Edyta Kania.
\newblock Semilinear {D}irichlet problem for the fractional {L}aplacian.
\newblock {\em Nonlinear Anal.}, 193:111512, 20, 2020.

\bibitem{BFSST18}
D.~Bonheure, J.~F{\"o}ldes, E.~dos Santos, A.~Saldaña, and H.~Tavares.
\newblock Paths to uniqueness of critical points and applications to partial
  differential equations.
\newblock {\em Transactions of the American Mathematical Society},
  370(10):7081--7127, 2018.

\bibitem{BrandleColoradoPabloSanchez}
C.~Br{\"a}ndle, E.~Colorado, A.~de~Pablo, and U.~S{\'a}nchez.
\newblock A concave—convex elliptic problem involving the fractional
  laplacian.
\newblock {\em Proceedings of the Royal Society of Edinburgh Section A:
  Mathematics}, 143(1):39--71, 2013.

\bibitem{BrascoLindgrenParini}
L.~Brasco, E.~Lindgren, and E.~Parini.
\newblock The fractional cheeger problem.
\newblock {\em Interfaces and Free Boundaries}, 16(3):419--458, 2014.

\bibitem{CatrinaWangJDE1999}
F.~Catrina and Z.-Q. Wang.
\newblock Nonlinear elliptic equations on expanding symmetric domains.
\newblock {\em J. Differential Equations}, 156(1):153--181, 1999.

\bibitem{ChenLiOu}
W.~Chen, C.~Li, and B.~Ou.
\newblock Qualitative properties of solutions for an integral equation.
\newblock {\em arXiv preprint math/0307262}, 2003.

\bibitem{ChenLiOu2006}
W.~Chen, C.~Li, and B.~Ou.
\newblock Classification of solutions for an integral equation.
\newblock {\em Comm. Pure Appl. Math.}, 59(3):330--343, 2006.

\bibitem{DamascelliGrossiPacellaAIHP1999}
L.~Damascelli, M.~Grossi, and F.~Pacella.
\newblock Qualitative properties of positive solutions of semilinear elliptic
  equations in symmetric domains via the maximum principle.
\newblock In {\em Annales de l'Institut Henri Poincar{\'e} C, Analyse non
  lin{\'e}aire}, volume~16, pages 631--652. Elsevier, 1999.

\bibitem{DancerJDE1988}
E.~N. Dancer.
\newblock The effect of domain shape on the number of positive solutions of
  certain nonlinear equations.
\newblock {\em J. Differential Equations}, 74(1):120--156, 1988.

\bibitem{DancerJDE1990}
E.~N. Dancer.
\newblock The effect of domain shape on the number of positive solutions of
  certain nonlinear equations. {II}.
\newblock {\em J. Differential Equations}, 87(2):316--339, 1990.

\bibitem{DancerRockyMountain1995}
E.~N. Dancer.
\newblock On the uniqueness of the positive solution of a singularly perturbed
  problem.
\newblock {\em Rocky Mountain J. Math.}, 25(3):957--975, 1995.

\bibitem{DancerMa2003}
EN~Dancer.
\newblock Real analyticity and non-degeneracy.
\newblock {\em Mathematische Annalen}, 325(2):369--392, 2003.

\bibitem{DLS17}
J.~D{\'a}vila, L.~L{\'o}pez~R{\'\i}os, and Y.~Sire.
\newblock Bubbling solutions for nonlocal elliptic problems.
\newblock {\em Revista Matem{\'a}tica Iberoamericana}, 33(2):509--546, 2017.

\bibitem{demak}
F.~De~Marchis, M.~Grossi, I.~Ianni, and F.~Pacella.
\newblock Morse index and uniqueness of positive solutions of the lane-emden
  problem in planar domains.
\newblock {\em Journal de Math{\'e}matiques Pures et Appliqu{\'e}es},
  128:339--378, 2019.

\bibitem{AF20}
G.~Di~Fratta and A.~Fiorenza.
\newblock A short proof of local regularity of distributional solutions of
  poisson’s equation.
\newblock {\em Proceedings of the American Mathematical Society},
  148(5):2143--2148, 2020.

\bibitem{DiNezza}
E.~Di~Nezza, G.~Palatucci, and E.~Valdinoci.
\newblock Hitchhikers guide to the fractional sobolev spaces.
\newblock {\em Bulletin des sciences math{\'e}matiques}, 136(5):521--573, 2012.

\bibitem{EspositoMussoPistoiaJDE2006}
P.~Esposito, M.~Musso, and A.~Pistoia.
\newblock Concentrating solutions for a planar elliptic problem involving
  nonlinearities with large exponent.
\newblock {\em J. Differential Equations}, 227(1):29--68, 2006.

\bibitem{FJ15}
M.~M. Fall and S.~Jarohs.
\newblock Overdetermined problems with fractional {L}aplacian.
\newblock {\em ESAIM Control Optim. Calc. Var.}, 21(4):924--938, 2015.

\bibitem{FV14}
M.~M. Fall and E.~Valdinoci.
\newblock Uniqueness and nondegeneracy of positive solutions of
  {$(-\Delta)^su+u=u^p$} in {$\mathbb R^N$} when {$s$} is close to 1.
\newblock {\em Comm. Math. Phys.}, 329(1):383--404, 2014.

\bibitem{FallWethNonexistence}
M.~M. Fall and T.~Weth.
\newblock Nonexistence results for a class of fractional elliptic boundary
  value problems.
\newblock {\em Journal of Functional Analysis}, 263(8):2205--2227, 2012.

\bibitem{FallWethMonotonicity}
M.~M. Fall and T.~Weth.
\newblock Monotonicity and nonexistence results for some fractional elliptic
  problems in the half-space.
\newblock {\em Communications in Contemporary Mathematics}, 18(01):1550012,
  2016.

\bibitem{fallmorse}
M.M. Fall, P.A. Feulefack, R.Y. Temgoua, and T.~Weth.
\newblock Morse index versus radial symmetry for fractional dirichlet problems.
\newblock {\em Advances in Mathematics}, 384:107728, 2021.

\bibitem{BS20}
J.~Fern{\'a}ndez~Bonder and A.~Salort.
\newblock Stability of solutions for nonlocal problems.
\newblock {\em Nonlinear Analysis}, 200:112080, 2020.

\bibitem{FrankLenzmann2013}
R.~L. Frank and E.~Lenzmann.
\newblock Uniqueness of non-linear ground states for fractional {L}aplacians in
  {$\mathbb{R}$}.
\newblock {\em Acta Math.}, 210(2):261--318, 2013.

\bibitem{FS}
R.~L Frank, E.~Lenzmann, and L.~Silvestre.
\newblock Uniqueness of radial solutions for the fractional laplacian.
\newblock {\em Communications on Pure and Applied Mathematics},
  69(9):1671--1726, 2016.

\bibitem{GNN1979}
B.~Gidas, W.~M. Ni, and L.~Nirenberg.
\newblock Symmetry and related properties via the maximum principle.
\newblock {\em Comm. Math. Phys.}, 68(3):209--243, 1979.

\bibitem{GidasSpruckCPAM1981}
B.~Gidas and J.~Spruck.
\newblock Global and local behavior of positive solutions of nonlinear elliptic
  equations.
\newblock {\em Comm. Pure Appl. Math.}, 34(4):525--598, 1981.

\bibitem{GidasSpruckCPDE1981}
B.~Gidas and J.~Spruck.
\newblock A priori bounds for positive solutions of nonlinear elliptic
  equations.
\newblock {\em Comm. Partial Differential Equations}, 6(8):883--901, 1981.

\bibitem{GladialiGrossiPacellaShrikanthCalcVar2011}
F.~Gladiali, M.~Grossi, F.~Pacella, and P.~N. Srikanth.
\newblock Bifurcation and symmetry breaking for a class of semilinear elliptic
  equations in an annulus.
\newblock {\em Calc. Var. Partial Differential Equations}, 40(3-4):295--317,
  2011.

\bibitem{GS16}
A.~Greco and R.~Servadei.
\newblock Hopf's lemma and constrained radial symmetry for the fractional
  {L}aplacian.
\newblock {\em Math. Res. Lett.}, 23(3):863--885, 2016.

\bibitem{GrossiADE2000}
M.~Grossi.
\newblock A uniqueness result for a semilinear elliptic equation in symmetric
  domains.
\newblock {\em Adv. Differential Equations}, 5(1-3):193--212, 2000.

\bibitem{GuoLiPistoiaYanJDE2021}
Y.~Guo, B.~Li, A.~Pistoia, and S.~Yan.
\newblock The fractional {B}rezis-{N}irenberg problems on lower dimensions.
\newblock {\em J. Differential Equations}, 286:284--331, 2021.

\bibitem{HS21}
V.~Hern{\'a}ndez-Santamar{\'\i}a and A.~Salda{\~n}a.
\newblock Existence and convergence of solutions to fractional pure critical
  exponent problems.
\newblock {\em Advanced Nonlinear Studies}, 21(4):827--854, 2021.

\bibitem{JSW20}
S.~Jarohs, A.~Saldaña, and T.~Weth.
\newblock A new look at the fractional poisson problem via the logarithmic
  laplacian.
\newblock {\em Journal of Functional Analysis}, 279(11):108732, 2020.

\bibitem{JW14}
S.~Jarohs and T.~Weth.
\newblock Asymptotic symmetry for a class of nonlinear fractional
  reaction-diffusion equations.
\newblock {\em Discrete \& Continuous Dynamical Systems}, 34(6):2581, 2014.

\bibitem{JinLiXiong}
T.~Jin, YY~Li, and J.~Xiong.
\newblock On a fractional nirenberg problem, part i: blow up analysis and
  compactness of solutions.
\newblock {\em Journal of the European Mathematical Society}, 16(6):1111--1171,
  2014.

\bibitem{KawohlLectureNotes1985}
B.~Kawohl.
\newblock {\em Rearrangements and convexity of level sets in {PDE}}, volume
  1150 of {\em Lecture Notes in Mathematics}.
\newblock Springer-Verlag, Berlin, 1985.

\bibitem{KwongLiTAMS1992}
M.~K. Kwong and Y.~Li.
\newblock Uniqueness of radial solutions of semilinear elliptic equations.
\newblock {\em Trans. Amer. Math. Soc.}, 333(1):339--363, 1992.

\bibitem{LiJDE1990}
Y.~Y. Li.
\newblock Existence of many positive solutions of semilinear elliptic equations
  on annulus.
\newblock {\em J. Differential Equations}, 83(2):348--367, 1990.

\bibitem{LiJEMS2004}
Y.Y. Li.
\newblock Remark on some conformally invariant integral equations: the method
  of moving spheres.
\newblock {\em J. Eur. Math. Soc. (JEMS)}, 6(2):153--180, 2004.

\bibitem{LinMM1994}
C.-S. Lin.
\newblock Uniqueness of least energy solutions to a semilinear elliptic
  equation in {$\mathbb{R}^2$}.
\newblock {\em manuscripta mathematica}, 84(1):13--19, 1994.

\bibitem{LongYanYangJDE2019}
W.~Long, S.~Yan, and J.~Yang.
\newblock A critical elliptic problem involving fractional {L}aplacian operator
  in domains with shrinking holes.
\newblock {\em J. Differential Equations}, 267(7):4117--4147, 2019.

\bibitem{McKennaPacellaPlumRothJDE2009}
P.~J. McKenna, F.~Pacella, M.~Plum, and D.~Roth.
\newblock A uniqueness result for a semilinear elliptic problem: a
  computer-assisted proof.
\newblock {\em J. Differential Equations}, 247(7):2140--2162, 2009.

\bibitem{McKennaPacellaPlumRoth2012}
P.~J. McKenna, F.~Pacella, M.~Plum, and D.~Roth.
\newblock A computer-assisted uniqueness proof for a semilinear elliptic
  boundary value problem.
\newblock In {\em Inequalities and applications 2010}, volume 161 of {\em
  Internat. Ser. Numer. Math.}, pages 31--52. Birkh\"{a}user/Springer, Basel,
  2012.

\bibitem{NiNussbaumCPAM1985}
W.-M. Ni and R.~D. Nussbaum.
\newblock Uniqueness and nonuniqueness for positive radial solutions of
  {$\Delta u+f(u,r)=0$}.
\newblock {\em Comm. Pure Appl. Math.}, 38(1):67--108, 1985.

\bibitem{Ou}
B.~Ou.
\newblock Positive harmonic functions on the upper half space satisfying a
  nonlinear boundary condition.
\newblock {\em Differential and Integral equations}, 9(5):1157--1164, 1996.

\bibitem{reguros}
X.~Ros-Oton and J.~Serra.
\newblock The dirichlet problem for the fractional laplacian: regularity up to
  the boundary.
\newblock {\em Journal de Math{\'e}matiques Pures et Appliqu{\'e}es},
  101(3):275--302, 2014.

\bibitem{RosOtonSerraARMA2014}
X.~Ros-Oton and J.~Serra.
\newblock The {P}ohozaev identity for the fractional {L}aplacian.
\newblock {\em Arch. Ration. Mech. Anal.}, 213(2):587--628, 2014.

\bibitem{ROSV17}
X.~Ros-Oton, J.~Serra, and E.~Valdinoci.
\newblock Pohozaev identities for anisotropic integrodifferential operators.
\newblock {\em Comm. Partial Differential Equations}, 42(8):1290--1321, 2017.

\bibitem{ServadeiValdinociCPAA2013}
R.~Servadei and E.~Valdinoci.
\newblock A {B}rezis-{N}irenberg result for non-local critical equations in low
  dimension.
\newblock {\em Commun. Pure Appl. Anal.}, 12(6):2445--2464, 2013.

\bibitem{ServadeiValdinociDCDS2013}
R.~Servadei and E.~Valdinoci.
\newblock Variational methods for non-local operators of elliptic type.
\newblock {\em Discrete Contin. Dyn. Syst.}, 33(5):2105--2137, 2013.

\bibitem{ServadeiValdinociTAMS2015}
R.~Servadei and E.~Valdinoci.
\newblock The {B}rezis-{N}irenberg result for the fractional {L}aplacian.
\newblock {\em Trans. Amer. Math. Soc.}, 367(1):67--102, 2015.

\bibitem{silver}
L.~Silvestre.
\newblock Regularity of the obstacle problem for a fractional power of the
  laplace operator.
\newblock {\em Communications on Pure and Applied Mathematics: A Journal Issued
  by the Courant Institute of Mathematical Sciences}, 60(1):67--112, 2007.

\bibitem{SrikanthDiffIntEq1993}
P.~N. Srikanth.
\newblock Uniqueness of solutions of nonlinear {D}irichlet problems.
\newblock {\em Differential Integral Equations}, 6(3):663--670, 1993.

\bibitem{ZhangCPDE1992}
L.~Q. Zhang.
\newblock Uniqueness of positive solutions of {$\Delta u+u+u^p=0$} in a ball.
\newblock {\em Comm. Partial Differential Equations}, 17(7-8):1141--1164, 1992.

\bibitem{ZouPisa1994}
H.~Zou.
\newblock On the effect of the domain geometry on uniqueness of positive
  solutions of {$\Delta u+u^p=0$}.
\newblock {\em Ann. Scuola Norm. Sup. Pisa Cl. Sci. (4)}, 21(3):343--356, 1994.

\end{thebibliography}
\end{document}